\documentclass[journal,twocolumn]{IEEEtran}
\usepackage{hyperref}
\usepackage{url}
\usepackage{threeparttable}
\usepackage{graphics} 
\usepackage{times} 
\usepackage{booktabs}       
\usepackage{amsfonts}       
\usepackage{nicefrac}       
\usepackage{microtype}      
\usepackage{algorithm}
\usepackage{algorithmicx}
\usepackage{algpseudocode}
\usepackage{mathtools}
\usepackage{graphicx}
\usepackage{cases}
\usepackage{array}
\usepackage{color}
\usepackage[cmyk]{xcolor}
\usepackage{float}
\usepackage{amssymb}
\usepackage{wrapfig}
\usepackage{subfigure}
\usepackage{amsmath}
\usepackage{soul}
\usepackage{amsthm}
\usepackage{graphicx} 
\usepackage{subfigure}
\usepackage{epstopdf}
\usepackage{multirow}
\usepackage{mathrsfs}
\usepackage{tikz}
\usepackage{pifont}
\usepackage[numbers,sort&compress]{natbib}

\theoremstyle{definition}
\newtheorem{definition}{Definition}

\theoremstyle{remark}
\newtheorem{theorem}{\indent Theorem}
\newtheorem{lemma}{\indent Lemma}

\newtheorem{corollary}{\indent Corollary}
\newtheorem{remark}{\indent Remark}
\newtheorem{proposition}{\indent Proposition}

\begin{document}

\title{Safety-critical Control with Control Barrier Functions: A Hierarchical Optimization Framework}


\author{Junjun~Xie$^1$,
        ~Liang~Hu$^1$, ~Jiahu~Qin$^2$,
        ~Jun Yang$^3$,
        and ~Huijun Gao$^4$

\thanks{$^1$ J. Xie and L. Hu are with the Department
of Automation, School of Mechanical Engineering and Automation, Harbin Institute of Technology, Shenzhen 518055,
China (e-mail:  23s053076@stu.hit.edu.cn;~l.hu@hit.edu.cn).}
\thanks{$^2$ J. Qin is with the Department of Automation,
University of Science and Technology of China,
Hefei 230085, China (e-mail: jhqin@ustc.edu.cn).}
\thanks{$^3$ J. Yang is with the Department of Aeronautical and Automotive Engineering, Loughborough University, Leicestershire, LE11 3TU, UK (e-mail:
j.yang3@lboro.ac.uk). }
\thanks{$^4$ H. Gao is with the Research Institute of Intelligent Control and Systems, Harbin Institute of Technology, Harbin 150001, China (e-mail: hjgao@hit.edu.cn}

}


\maketitle

\begin{abstract}
The control barrier function (CBF) has become a fundamental tool in safety-critical systems design since its invention. Typically, the quadratic optimization framework is employed to  accommodate CBFs,  control Lyapunov functions (CLFs), other constraints and nominal control design. However, the constrained optimization framework involves hyper-parameters to tradeoff different objectives and constraints, which, if not well-tuned beforehand,  impact system performance and even lead to infeasibility. In this paper, we propose a hierarchical optimization framework that decomposes the multi-objective optimization problem into nested optimization sub-problems in a safety-first approach. The new framework addresses potential infeasibility on the premise of ensuring safety and performance as much as possible and applies easily in multi-certificate cases. With vivid visualization aids, we systematically analyze the advantages of our proposed method over existing QP-based ones in terms of safety, feasibility and convergence rates. Moreover, two numerical examples are provided that verify our analysis and show the superiority of our proposed method.
\end{abstract}


\section{Introduction}\label{sec:intro}

As autonomous systems such as drones and self-driving vehicles become more prevalent in our societies, concerns about their safety, e.g., collisions with humans, 
have also grown. The potential for catastrophic outcomes due to incorrect operations of autonomous systems underscores the urgency of addressing safety issues. As such, significant research efforts have been directed toward developing novel safety-critical control design paradigms, where a unified system synthesis framework accommodates both metrics in traditional control system design such as stability and convergence rates, and safety concerns.


Recently, control barrier functions (CBFs) have been introduced as a promising method to accommodate safety specification into control design. \cite{acc_CDC,QP_TAC} presents a quadratic program (QP) framework that combines CBFs with control Lyapunov functions (CLFs), which becomes a basic design paradigm in safety-critical control. Based on it, various robust and adaptive CBF variants that consider model uncertainty and disturbance have been proposed \cite{adaptiveCBF_ACC,adaptiveCBF_XW,robustCBF_YH,DOBCCBF_TAC}. Besides, nonsmooth control barrier functions are proposed to compose multiple set-based constraints (e.g., connectivity and collision avoidance) \cite{nonsmoothcbf_tro,nonsmoothcbf_cdc,nonsmoothcbf_ral,nonsmoothcbf_ccta}. On the other hand, machine learning techniques have been utilized to synthesize learning-based CBF control in a model-free manner \cite{RLCBF_RSS,barriernet_TRO,expertdata_CDC,survey_TRO,S2RL_TIV} and deployed on robots and autonomous vehicles for verification \cite{cbf4rob_icra1,cbf4rob_ral1,cbf4rob_icra2,cbf4rob_icra3}.

A major challenge facing safety-critical control is how to balance safety, system performance and solution feasibility. The classic CBF based QPs method proposed in \cite{acc_CDC,QP_TAC} formulate safety requirements as a set of hard inequality constraints with manually specified performance indexes such as decay rate of CLFs and CBFs. Without fine-tuned performance indexes, the constrained quadratic problems could become infeasible unexpectedly, which is unacceptable for applications in safety-critical systems. To address the above issue, \cite{optimal-decay_ACC} proposed an improved method called Optimal-decay CLF-CBF QP, in which the hyper-parameters in CBF constraint are optimized rather than fixed ad hoc or handcrafted. Owing to the introduced extra slack variables, the Optimal-decay CLF-CBF QP is guaranteed point-wise feasible at the expense of relaxing CBF constraints into soft constraints, which means safety constraints are no longer strictly satisfied. \cite{sufficient_Automatica} provides sufficient conditions for feasibility using a dedicate CBF to describe the input constraints, but it is quite complicated to design such a CBF. 
\cite{checking_CDC2022} proposes an approach to check the compatibility between constraints, and \cite{2024_Automatica,withconstraints_LCSS2022} provide algorithms to address the conflict-constraints cases, all of which involve additional hand-crafted parameters. 

To overcome the above-mentioned shortcomings in existing safety-critical design methods, we propose a safety-first approach that prioritizes safety in the formulated optimization problems but without being over-conservative in safety. Different from the existing CLF-CBF QP framework, our safety-first approach decomposes the entire multi-objective optimization problem hierarchically to three nested optimization problems. Specifically, safety requirement is certified first, then stability and the last performance are optimized sequentially. If there does not exist feasible control input that satisfies safety specifications, the approximate solution with the least violation in safety is sought automatically, eliminating solution infeasibility. In addition, such a hierarchical optimization framework is flexible and easily extended to the cases with multiple safety certificates, e.g., collision avoidance with multiple obstacles. 

Moreover, we introduce a new concept sub-safe which extends the safe concept introduced in existing CBF based QPs methods \cite{acc_CDC,QP_TAC} by considering system safety with unsafe initial conditions. As a result, we design the safety-first CLF-CBF QP method such that not only the system with a safe initial condition will stay in the safe set, but also the system with an unsafe initial condition will be driven into the safe set if it is achievable.

The main contributions of this paper are summarized as follows:
\begin{enumerate}
    \item \textit{We propose a novel safety-first hierarchical optimization approach to safe-critical control design, which eliminates solution infeasibility, optimize system performance while respects safety with the higher priority. We also extend it to the more general multi-certificates problems with objectives of different levels of priorities, showing the flexibility of our proposed framework;}
    \item \textit{We provide a systematic analysis together with straightforward geometric visualization that reveals the cons of existing methods in handling feasibility, safety and convergence and pros of ours, as well as the relations between them;}
    \item \textit{Numerical simulations of two representative examples are provided with detailed comparative analysis, showing the superiority of our methods over the SOTA ones.}
\end{enumerate}

This paper is organized as follows. \autoref{sec: Preliminaries} introduce
the mathematical preliminaries used later and formulate the problem. Then in \autoref{sec:safetyfirst}, our novel method called Safety-first CLF-CBF QP is presented in steps and extended to multi-certificates cases. \autoref{sec:unified} gives two unified forms for all existing QP based methods and \autoref{sec:shortcomings} analyzes feasibility, safety and performance in the unified forms to reveals the cons of existing methods and pros of ours. Two representative examples are provided in \autoref{sec:simulation} to verify our analysis. Finally, \autoref{sec:conclusion} provides a summary and conclusion of this paper.
\\\\
\textit{Notation: }$\partial \mathcal{C}$ and $\operatorname{Int}(\mathcal{C})$ stand for the interior and boundary of the set $\mathcal{C}$. $\Vert\cdot\Vert$ denotes the Euclidean norm of the vector. $ x^TP(\cdot)$ denotes the quadratic form $x^TPx$ for shortness. $\nabla(\cdot)$ denotes the gradient of the function. A continuous function $\alpha:\mathbb{R}\rightarrow\mathbb{R}$ belongs to extended $\mathcal{K_\infty}$ functions if $\alpha(0)=0,\lim\limits_{r\to +\infty}\alpha(r)=+\infty,\lim\limits_{r\to -\infty}\alpha(r)=-\infty$ and it is strictly increasing. $\text{diag}\left[\begin{array}{cccc}a&b&\cdots
\end{array}\right]$ denotes the diagonal matrix $\left[\begin{array}{cccc}a\\&b\\&&\ddots
\end{array}\right]$ for shortness.

\section{Preliminaries and Problem Formulations}\label{sec: Preliminaries}

\subsection{Preliminaries}\label{subsec: Preliminaries}
In this section, we briefly introduce the concepts of CBFs, CLFs and two classic control design methods that integrate CBFs and CLFs in a QP framework.

Consider a control-affine system
\begin{equation}
    \label{model0}\dot{x}=f(x)+g(x)u,\ x(t_0)=x_0
\end{equation}
where $x\in \mathbb{R}^n$, $u\in \mathbb{R}^m$, and $f:\mathbb{R}^n\rightarrow\mathbb{R}^n$, $g:\mathbb{R}^n\rightarrow\mathbb{R}^{n\times m}$ are locally Lipschitz. Assume that a nominal feedback controller is used to stabilize the system:
\begin{equation}
   u=k(x),\ u\in \mathcal{U}\label{inputcons0}
\end{equation}
where the input constraint $\mathcal{U}$ is a convex polytope. 

\begin{definition}\label{Def_CLF}
    \cite{ESCLF_TAC}A continuously differentiable function $V:\mathbb{R}^n\rightarrow\mathbb{R}$ is an \textbf{exponentially stabilizing control Lyapunov function (ES-CLF)} if there exist positive constants $c_1,c_2,c_3 > 0 $ for all $x\in \mathbb{R}^n$ such that 
    \begin{gather}
        c_1\Vert x\Vert^2\leq V(x)\leq c_2\Vert x\Vert^2,
        \\
        \label{CLF_constraint}
        \inf _{u \in U}\left[L_{f} V(x)+L_{g} V(x) u+c_{3} V(x)\right] \leq 0.
    \end{gather}
    where $L_fV(x)\triangleq\nabla^T{V(x)} f(x)\in \mathbb{R}$ and $L_gV(x)\triangleq\nabla^T{V(x)} g(x)\in\mathbb{R}^{1\times m}$ are Lie-derivatives of $V(x)$ along $f(x)$ and $g(x)$, respectively.
\end{definition}
\begin{definition}\label{Def_CBF}
    \cite{theoryandapp_ECC} Assume a closed set $\mathcal{C}$ is defined as
    \begin{equation}\label{safeset}
        \begin{aligned}
            \mathcal{C} & =\left\{x \in \mathbb{R}^{n}: h(x) \geq 0\right\}, \\
            \partial \mathcal{C} & =\left\{x \in \mathbb{R}^{n}: h(x)=0\right\}, \\
            \operatorname{Int}(\mathcal{C}) & =\left\{x \in \mathbb{R}^{n}: h(x)>0\right\},
        \end{aligned}
    \end{equation}
    where $h:\mathbb{R}^n\rightarrow\mathbb{R}$ is a continuously differentiable function.  The function $h(x)$ is called a \textbf{zeroing control barrier function (ZCBF)} if there exists an extended $\mathcal{K_\infty}$ function $\alpha$ such that
    \begin{equation}\label{CBF_constraint}
        \sup _{u \in U}\left[L_{f} h(x)+L_{g} h(x) u+\alpha(h(x))\right] \geq 0.
    \end{equation}
        where $L_fh(x)\triangleq\nabla^T{h(x)} f(x)\in\mathbb{R}$ and $L_gh(x)\triangleq\nabla^T{h(x)}g(x)\in \mathbb{R}^{1\times m}$ are Lie-derivatives of $h(x)$ along $f(x)$ and $g(x)$, respectively, and the set $\mathcal{C}$ is called a safe set.
\end{definition}

\begin{remark}
    Another form of control barrier function called reciprocal control barrier functions (RCBF) can be converted to ZCBF under some mild conditions \cite{QP_TAC}. Compared to RCBF, ZCBF is more concise and widely used. Hence, in this paper all our analysis and discussions  are in the form of ZCBF, referred to as CBF for shortness. Additionally, by choosing a special case of the extended $\mathcal{K_\infty}$ function $\alpha$, that is $\alpha(h(x))=\gamma h(x)$ with a scalar $\gamma>0$, it makes \eqref{CBF_constraint} become an inequality in the similar form of \eqref{CLF_constraint}.
\end{remark}

\begin{definition}\label{def:safety}\cite{QP_TAC}\cite{theoryandapp_ECC}
    Given the system \eqref{model0}-\eqref{inputcons0} with a safe set $\mathcal{C}$ defined in \eqref{safeset}, the system is safe under the control $u$ if 
    \begin{equation}\label{op:safety}
        \forall x_0\in \mathcal{C}\quad\Rightarrow\quad x(t)\in\mathcal{C}, \forall t\in [0,\tau_{max})
    \end{equation}
    i.e., $\mathcal{C}$ is forward invariant.
\end{definition}

 Control barrier functions 
 provide a formal approach to safe control design. Given the system \eqref{model0}-\eqref{inputcons0}, if the condition
\begin{equation}\label{CBFperfect}
    L_{f} h(x)+L_{g} h(x) u+\gamma h(x) \geq 0
\end{equation}
is always satisfied, then the closed-loop system is safe all the time.  

Similarly, control Lyapunov functions are used for stable controller design. If the condition
\begin{equation}\label{CLFperfect}
    L_{f} V(x)+L_{g} V(x) u+\lambda V(x) \leq 0
\end{equation}
is always satisfied, then the closed-loop system is stable all the time. Since the CBF and CLF are bounded by exponential functions with $\gamma$ and $\lambda$, respectively, they are called \textbf{decay rates}, which characterize closed-loop transient dynamics.  

Motivated by the above results \eqref{CBFperfect}-\eqref{CLFperfect}, a naive idea of control design using CLFs and CBFs can be expressed as 
\begin{subequations}
    \begin{align}
    \label{object:Hard CLF-CBF QP}
    u^*&=\mathop{\arg\min}\limits_{u}\frac{1}{2}(u-k(x))^TH(\cdot)\\
        \label{cons:HardCLF} \textbf{s.t.} \quad\quad &L_{f} V(x)+L_{g} V(x) u+\lambda V(x) \leq 0\\
        \label{cons:HardCBF}&L_{f} h(x)+L_{g} h(x) u+\gamma h(x) \geq 0\\
        \label{cons:HardInput}& u\in \mathcal{U}
    \end{align}
\end{subequations}
since both \eqref{cons:HardCLF} and \eqref{cons:HardCBF} are hard constraints,  we call it \textbf{Hard CLF-CBF QP}.

Unfortunately, in practice the two constraints \eqref{cons:HardCLF} and \eqref{cons:HardCBF} may conflict, resulting in no feasible solutions. In addition, if consider the input constraints  \eqref{cons:HardInput} as well, there will be a higher chance of conflicts among them and thus of no feasible solutions.

To address the above infeasibility issues, the \textbf{CLF-CBF QP} method was proposed in \cite{QP_TAC}, which is summarised as below:
\begin{subequations}
    \begin{align}
    \label{QP_model}[u^*,\delta^*]=&\mathop{\arg\min}\limits_{[u,\delta]}\frac{1}{2}(u-k(x))^TH(\cdot)+p\delta^2\\
        \label{CLFcons} \textbf{s.t.} \quad\quad &L_{f} V(x)+L_{g} V(x) u+\lambda V(x) \leq \delta\\
        \label{CBFcons}&L_{f} h(x)+L_{g} h(x) u+\gamma h(x) \geq 0\\
        \label{inputcons}& u\in \mathcal{U}
    \end{align}
\end{subequations}
where  $\delta$ is a slack variable to make the CLF constraint \eqref{CLFcons} as a soft constraint,  mitigating its conflict with the CBF (hard) constraint \eqref{CBFcons}, and the positive definite symmetric matrix $H$ and the weighting coefficient $p$ are hyperparameters to be tuned beforehand. 

As analyzed in \cite{optimal-decay_ACC}, the above optimization problem may become infeasible without a fine-tuning hyper-parameter. Motivated by it, an improved method called \textbf{Optimal-decay CLF-CBF QP}  was proposed in \cite{optimal-decay_ACC} as below:

\begin{subequations}
{
    \begin{align}
    \label{opdecay_objective}[u^*,\delta^*,\omega^*]=&\mathop{\arg\min}\limits_{[u,\delta,\omega]}\frac{1}{2}(u-k(x))^TH(\cdot)+p\delta^2+p_\omega(\omega-\omega_0)^2\\
        \label{CLFcons2} \textbf{s.t.}\quad\quad&L_{f} V(x)+L_{g} V(x) u+\lambda V(x) \leq \delta\\
        \label{CBFcons2}&L_{f} h(x)+L_{g} h(x) u+\omega\gamma_0 h(x) \geq 0\\
        \label{inputcons2}&u\in \mathcal{U}
    \end{align}}
\end{subequations}
where two new parameters $p_\omega$ and $\omega_0$ are introduced to ensure the feasibility of the QP. Unfortunately, as mentioned in \cite{optimal-decay_ACC}, the safety of the system is no longer strictly guaranteed after \eqref{CBFcons} is relaxed into \eqref{CBFcons2}.


\subsection{Problem formulation}\label{subsec: problemformulation}
Before presenting problem formulation, we would like to introduce a new concept on safety.
\begin{definition}\label{def:subsafety}
    Given the system \eqref{model0}-\eqref{inputcons0} with a safe set $\mathcal{C}$ defined in \eqref{safeset}, the system is \textbf{\textit{sub-safe}} under the control $u$ if 
    \begin{equation}\label{eq:subsafety}
        \forall x_0\notin \mathcal{C},\quad x(t) \in \mathcal{C} \quad \text{when} \quad t \rightarrow \infty.
    \end{equation}
\end{definition}
\autoref{def:safety} only considers the cases $x_0\in \mathcal{C}$. In practice, safe sets could be designed conservatively for easy computation and hence the initial state of system under consideration is unsafe. A motivating example is, using a safe distance $D>0$ as the safe set to avoid collision in adaptive cruise control \cite{acc_CDC}. If the initial distance $d_0$ is small and $d_0<D$, it is "unsafe" according to the \autoref{def:safety}. When dealing with unsafe initial conditions , the existing QP-based methods could easily become infeasible, or lose safety guarantees, as shown in \autoref{fig3}. Motivated by this, we develop the concept of sub-safe to consider systems with unsafe conditions $h(x)<0$. 

\begin{figure}
    \centering \includegraphics[scale=0.35]{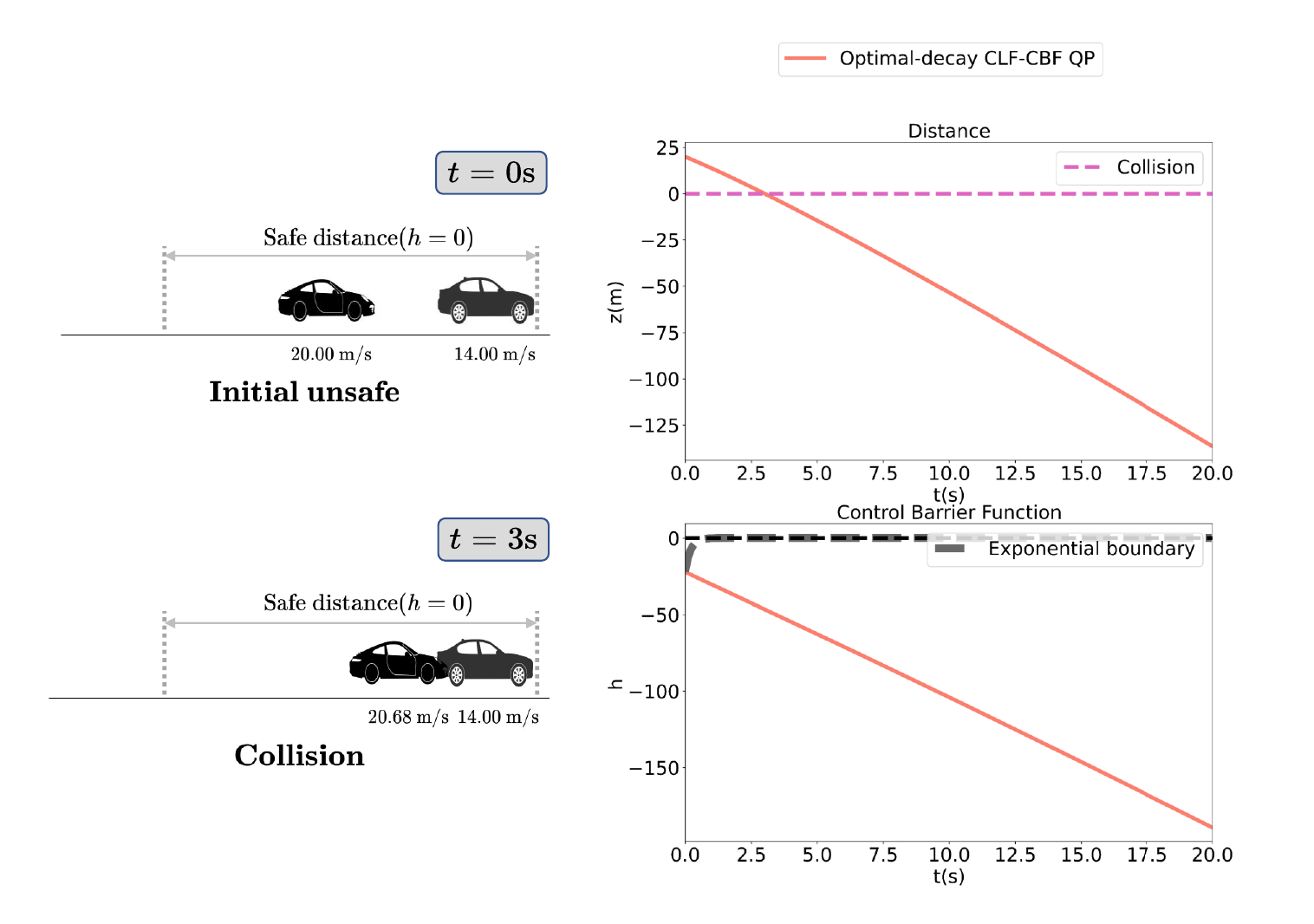}
    \vspace{-0.8cm}
    \caption{A motivating example for sub-safe.  With an unsafe initial distance between two vehicles, the existing method CLF-CBF QP \cite{QP_TAC} is infeasible while Optimal-decay CLF-CBF QP \cite{optimal-decay_ACC} drives the vehicle to speed up and collide. To consider such unsafe situations, we develop the new concept of sub-safe.}
    \label{fig3}
\end{figure}

The problem we will investigate in this paper is formally formulated as below:

 Given the system \eqref{model0}-\eqref{inputcons0} with a safe set $\mathcal{C}$ defined in \eqref{safeset}, we aim to present a new optimization framework to calculate the control sequence $u(t)$ such that
\begin{enumerate}
    \item (\textbf{Feasibility}) the optimization problem is always feasible, i.e. there always exist a control input at any state $x$;
    \item (\textbf{Safety}) the control input guarantees both the safety and sub-safety of the system as much as possible;
    \item (\textbf{Convergence}) the closed-loop system converges with prescribed decay rates in both safety and stability if achievable.
\end{enumerate}

\section{Safety-first CLF-CBF QP: a hierarchical optimization framework}\label{sec:safetyfirst}
In this section, we propose a new safety-first CLF-CBF QP framework and highlight its advantages. We then extend this framework to address broader scenarios involving multiple safety constraints simultaneously, such as navigation with collision avoidance of multiple obstacles.
\subsection{Safety-first Approach}\label{subsec:steps}

\begin{algorithm}[t]
\caption{Safety-first CLF-CBF QP}
\label{alg1}
\begin{algorithmic}[1]
\State Given the system \eqref{model0}, input constraint \eqref{inputcons0}, the nominal controller $k(x)$, Choose a control Lyapunov function $V$ and a control barrier function $h$. 
\State Set a positive definite weight matrix $H$ of input $u$ and decay rates $\lambda$, $\gamma$ of CLF, CBF constraints.
\State \% CBF constraint minimal-slacked
\State Solve \eqref{op:safetyfirst_1}
\State $\delta_2^*\leftarrow$ solution of \eqref{op:safetyfirst_1}
\State \% CLF constraint minimal-slacked
\State Solve \eqref{op:safetyfirst_2}
\State $\delta_1^*\leftarrow$ solution of \eqref{op:safetyfirst_2}
\State \% Input-norm minimized
\State Solve \eqref{op:safetyfirst_3}
\State $u^*\leftarrow$ solution of \eqref{op:safetyfirst_3}
\end{algorithmic}
\end{algorithm}

The basic idea is to set different levels of priorities between stability, safety and performance in the proposed optimization framework. Specifically, safety is always prioritized over stability, and then stability over performance. Hence, we propose to design the controller by considering the three requirements hierarchically rather than simultaneously.  Due to this, we call our method \textbf{\textit{safety-first CLF-CBF QP}}. The entire framework is composed of three nested sub-problems: sub-problem 1 $\xrightarrow{\delta_2^*}$ sub-problem 2 $\xrightarrow{\delta_1^*}$ sub-problem 3 $\rightarrow u^*$, and the three sub-problems are presented below:

First, consider the sub-problem 1
\begin{equation}
    \begin{aligned}
    \label{op:safetyfirst_1}   \textbf{Sub-problem 1} \quad  &[*,*,\delta_2^*]=\mathop{\arg\min}\limits_{[u,\delta_1,\delta_2]}\delta_2^2\\
    \textbf{s.t.}\quad\quad&L_{f} V(x)+L_{g} V(x) u+\lambda V(x) \leq \delta_1\\
        &L_{f} h(x)+L_{g} h(x) u+\gamma h(x) \geq \delta_2\\
        &u\in\mathcal{U}
    \end{aligned}
\end{equation}
Note that only the slack variable of CBF constraint ($\delta_2$) is minimized by solving the above sub-problem 1. Next, we substitute $\delta_2^*$ into the sub-problem 2 below:
\begin{equation}
    \begin{aligned}
    \label{op:safetyfirst_2}     \textbf{Sub-problem 2} \quad   &[*,\delta_1^*]=\mathop{\arg\min}\limits_{[u,\delta_1]}\delta_1^2\\
    \textbf{s.t.}\quad\quad&L_{f} V(x)+L_{g} V(x) u+\lambda V(x) \leq \delta_1\\
        &L_{f} h(x)+L_{g} h(x) u+\gamma h(x) \geq \delta_2^*\\
        &u\in\mathcal{U}
    \end{aligned}
\end{equation}
Since $\delta^*_2$ is fixed as the one obtained by solving the sub-problem 1, now only the slack variable of CBF constraint ($\delta_1$) is minimized by solving the above sub-problem 2. Finally, substitute $\delta_1^*$ into the sub-problem 3 and we have
\begin{equation}
    \begin{aligned}
    \label{op:safetyfirst_3}      \textbf{Sub-problem 3} \quad &u^*=\mathop{\arg\min}\limits_{u}(u-k(x))^TH(\cdot)\\
    \textbf{s.t.}\quad\quad&L_{f} V(x)+L_{g} V(x) u+\lambda V(x) \leq \delta_1^*\\
        &L_{f} h(x)+L_{g} h(x) u+\gamma h(x) \geq \delta_2^*\\
        &u\in\mathcal{U}
    \end{aligned}
\end{equation}
The solution $u^*$ is the control inputs. 

The above three sub-problems are feasible all the time, which will be proved later in \autoref{theorem:feasibility} later in  \autoref{subsec:feasiblity}.

\subsection{Multi-certificates Optimal Control Problems}\label{subsec:multi-objective}
The idea of hierarchical optimization can be extended into multi-objective cases naturally. If consider more metrics about the system such as connectivity \cite{nonsmoothcbf_cdc} or multi-obstacle in collision avoidance, we need more certificates such as CLFs or CBFs to guarantee the system. All the cases can be formulated, unified and addressed as follows. Assume there exists $n$ metrics being considered.
\begin{enumerate}
    \item Certificates construction. Construct certificates for each metric which is a mapping such as $\mathscr{C}:x\mapsto \mathscr{C}(x)\in \mathbb{R}$.
    \item Priorities sorting. Sort all the metrics into $k$ groups in order of priorities and the $i_{th}$ group consists of $m_i$ metrics. Then rename $\mathscr{C}$ as $\mathscr{C}_{ij}$, where $i$ denotes the priority level and $j=1,\ldots, m_j$ and $c_{ij}$ denotes its weight in the same level. A priority list is illustrated in  \autoref{fig:Priorities-sorting}. For example, the metric $\mathscr{C}_{21}$ is in level 2 with weight $c_{21}$, so $\mathscr{C}_{21}$ is prior to $\mathscr{C}_{11}$. Assume $c_{21}>c_{22}$, it means that $\mathscr{C}_{21}$ is not prior but with bigger weighting than $\mathscr{C}_{22}$.
    \item Solving the QP problems. Using the priority list, we can design a series of QP sub-problems which are
    \begin{equation*}
    \begin{aligned}
\textbf{Sub-problem 1}\quad&[*,\delta_{k1}^*,...,\delta_{km_k}^*]=\mathop{\arg\min}\sum\limits_{1\leq j\leq m_k} c_{kj}\delta_{kj}^2\\
 \textbf{s.t.}\quad\quad
&L_{f}\mathscr{C}_{ij}+L_{g}\mathscr{C}_{ij}u+r_{ij} \mathscr{C}_{ij}\leq( \text{or} \geq) \delta_{ij},\\ &\quad 1\leq i\leq k,1\leq j\leq m_i\\
&u\in \mathcal{U}
    \end{aligned}
    \end{equation*}
Substitute the optimal solution to the next sub-problem, and repeat the procedure until reaching the lowest level of priority:
\begin{equation*}
    \begin{aligned}
\textbf{Sub-problem k}\quad&
[*,\delta_{11}^*,...,\delta_{1m_{1}}^*]=\mathop{\arg\min}\sum\limits_{1\leq j\leq m_1} c_{1j}\delta_{1j}^2\\
 \textbf{s.t.}\quad\quad
&L_{f}\mathscr{C}_{1j}+L_{g}\mathscr{C}_{1j}u+r_{1j} \mathscr{C}_{1j}\leq( \text{or} \geq) \delta_{1j},\\ & \quad  1\leq j\leq m_1\\
&L_{f}\mathscr{C}_{ij}+L_{g}\mathscr{C}_{ij}u+r_{ij} \mathscr{C}_{ij}\leq( \text{or} \geq) \delta^*_{ij},\\ & \quad   2\leq i\leq k,1\leq j\leq m_i\\
&u\in \mathcal{U}
    \end{aligned}
    \end{equation*}
Substitute $\delta^*_{ij}$ into the final sub-problem and we have
\begin{equation*}
    \begin{aligned}
\textbf{Sub-problem k+1} \quad 
& u^*=\mathop{\arg\min}\limits_{u} (u-k(x))^TH(\cdot)\\
 \textbf{s.t.}\quad\quad
&L_{f}\mathscr{C}_{ij}+L_{g}\mathscr{C}_{ij}u+r_{ij} \mathscr{C}_{ij}\leq( \text{or} \geq) \delta^*_{ij},\\ & \quad 1\leq i\leq k,1\leq j\leq m_i\\
&u\in \mathcal{U}
    \end{aligned}
    \end{equation*}
the final solution $u^*$ is the control input. 
\end{enumerate}
    
\begin{figure}
    \centering
    \includegraphics[scale=.87]{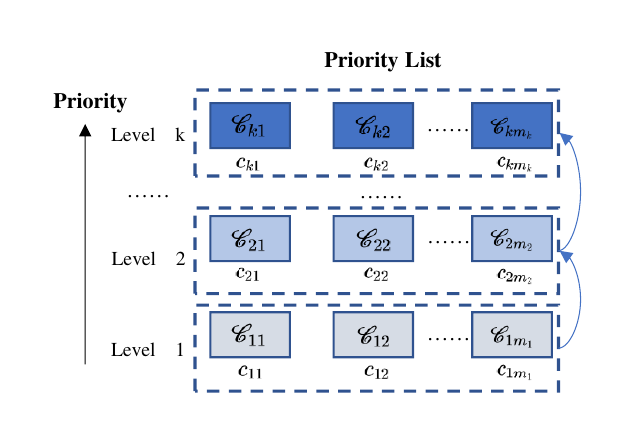}
    \caption{A priority list shows different priority of all the metrics considered. A higher level means that the metric has a higher priority, such as $\mathscr{C}_{21}$ is prior to $\mathscr{C}_{11}$. The weight $ c_{ij}$ denotes the relative importance among metrics within the same priority level.}
    \label{fig:Priorities-sorting}
\end{figure}
\section{The unified form of all CLF-CBF QP Frameworks}\label{sec:unified}



To facilitate comparative analysis among typical CLF-CBF QP methods and our proposed approach,  we first present a unified form for all three existing methods described in Section \ref{sec: Preliminaries}. Then we show that our safety-first CLF-CBF QP is equivalently a limit-weight form of the Unified CLF-CBF QP.

\subsection{Unified CLF-CBF QP}

\begin{lemma}\label{lemma1}
All of the CLF-CBF QP \eqref{QP_model}-\eqref{inputcons}, Optimal-decay CLF-CBF QP \eqref{opdecay_objective}-\eqref{inputcons2} and Hard CLF-CBF QP \eqref{object:Hard CLF-CBF QP}-\eqref{cons:HardInput} can be unified in the following form (called \textbf{Unified CLF-CBF QP} thereafter):
\begin{subequations}
    \begin{align}
    \label{unified_func}      [u^*,\Delta^*]=&\mathop{\arg\min}\limits_{[u,\Delta]}\frac{1}{2}
        \left[\begin{array}{cc}
             u-k(x)\\
             \Delta
        \end{array}\right]^T
        \left[\begin{array}{cc}
             H& 0 \\
             0&H_\Delta 
        \end{array}\right]
        \Big(\cdot\Big)\\
    \textbf{s.t.}\quad\quad\label{eq:CLFcons3}&L_{f} V(x)+L_{g} V(x) u+\lambda V(x) \leq \delta_1\\
        \label{CBFcons3}&L_{f} h(x)+L_{g} h(x) u+\gamma h(x) \geq \delta_2\\
        \label{inputcons3}&u\in\mathcal{U}\\
        \label{slackcons}&\Delta\in \mathcal{U}_\Delta
    \end{align}
\end{subequations}
where $\Delta=\left[\begin{array}{cc}\delta_1\\\delta_2
\end{array}\right]$, the diagonal matrix $H_\Delta\succ 0$, and $\mathcal{U}_\Delta \subset \mathbb{R}^2$.
\end{lemma}
\begin{proof}
Obviously, \eqref{inputcons3} is the same as \eqref{inputcons}, \eqref{inputcons2} and \eqref{cons:HardInput}. Hence what remains to prove is  how \eqref{unified_func}-\eqref{CBFcons3} reduce to \eqref{QP_model}-\eqref{CBFcons}, \eqref{opdecay_objective}-\eqref{CBFcons2} and \eqref{object:Hard CLF-CBF QP}-\eqref{cons:HardCBF} with different defined $\mathcal{U}_\Delta$ in \eqref{slackcons}, respectively, which is elaborated one by one as below:

1) \eqref{unified_func}-\eqref{CBFcons3} reduce to \eqref{QP_model}-\eqref{CBFcons}. Choose $\mathcal{U}_\Delta=(\mathbb{R},\{0\})$, i.e., set $\delta_2=0,\delta_1=\delta\in\mathbb{R}$, and $H_\Delta=\text{diag}
\left[\begin{array}{cc}2p&2\end{array}\right]$, then it is easily found that Unified CLF-CBF QP in \eqref{unified_func}-\eqref{slackcons} reduces to CLF-CBF QP.

2) \eqref{unified_func}-\eqref{CBFcons3} reduce to \eqref{opdecay_objective}-\eqref{inputcons2}. By setting $\delta_1=\delta\in\mathbb{R}$, $\Delta_\omega=\omega_0-\omega$, $\delta_2=\Delta_\omega\gamma_0 h(x)$ and $\gamma=\omega_0 \gamma_0$, obviously \eqref{CLFcons3} reduces to \eqref{CLFcons2}, and we have
\begin{align}
    &L_{f} h(x)+L_{g} h(x) u+\gamma h(x) \geq \delta_2\notag\\
    \Rightarrow \quad
    &L_{f} h(x)+L_{g} h(x) u+\omega_0\gamma_0 h(x) \geq \Delta_\omega\gamma_0 h(x)\notag\\
    \Rightarrow \quad
    &L_{f} h(x)+L_{g} h(x) u+[\omega_0-(\omega_0-\omega)]\gamma_0 h(x) \geq 0 \notag\\
    \Rightarrow \quad
    &L_{f} h(x)+L_{g} h(x) u+\omega\gamma_0 h(x) \geq 0 \notag
\end{align}
showing \eqref{CBFcons3} reduces to \eqref{CBFcons2}.

Next, consider \eqref{unified_func}. For the case $h(x)\neq0$, choose $\mathcal{U}_\Delta=\mathbb{R}^2$, take $p_h=\frac{p_\omega}{(\gamma_0 h(x))^2}$ and $H_\Delta=\text{diag}\left[\begin{array}{cc}
      2p&2p_h
\end{array}\right]$, then the quadratic forms on the right-hand side of \eqref{unified_func} can be rewritten as
\begin{align}
    &\frac{1}{2}(u-k(x))^TH(\cdot)+\frac{1}{2}\Delta^T H_\Delta\Delta\notag\\
    =&\frac{1}{2}(u-k(x))^TH(\cdot)+p\delta^2+p_h(\delta_2)^2\notag\\
    =&\frac{1}{2}(u-k(x))^TH(\cdot)+p\delta^2+\frac{p_\omega}{(\gamma_0 h(x))^2}(\delta_2)^2\notag\\
    =&\frac{1}{2}(u-k(x))^TH(\cdot)+p\delta^2+p_\omega(\frac{\delta_2}{\gamma_0 h(x)})^2\notag\\
    =&\frac{1}{2}(u-k(x))^TH(\cdot)+p\delta^2+p_\omega(\Delta_\omega)^2\notag\\
    =&\frac{1}{2}(u-k(x))^TH(\cdot)+p\delta^2+p_\omega(\omega-\omega_0)^2 \notag.
\end{align}
which is exactly \eqref{opdecay_objective}, and since $\mathcal{U}_\Delta=\mathbb{R}^2$, then $\omega=\omega_0-\frac{\delta_2}{\gamma_0 h(x)}\in \mathbb{R}$.

If $h(x)=0$, then $\delta_2=\Delta_\omega\gamma_0 h(x)=0$, which is equivalent to $\mathcal{U}_\Delta=\mathbb{R}\times\{0\}$, the same as CLF-CBF QP. Referring to \eqref{CBFcons2}, it is obvious that $\omega$ does not influence the feasible region of the problem. Therefore, to minminze \eqref{opdecay_objective}, $\omega=\omega_0$ will be always satisfied, then \eqref{opdecay_objective} becomes \eqref{QP_model}. As proved previously that \eqref{unified_func} reduces to \eqref{QP_model}, so it follows that \eqref{unified_func} reduces to \eqref{opdecay_objective}.

3) \eqref{unified_func}-\eqref{CBFcons3} reduce to \eqref{object:Hard CLF-CBF QP}-\eqref{cons:HardInput}. Choose $\mathcal{U}_\Delta=\{0\}\times\{0\}$, i.e., $\delta_1$ and $\delta_2$ are fixed as $0$, then \eqref{CLFcons3} and \eqref{CBFcons3} reduce to \eqref{cons:HardCLF} and \eqref{cons:HardCBF} respectively. Set $H_\Delta=\text{diag}\left[\begin{array}{cc}2&2
\end{array}\right]$, and since $\Delta=0$, then $\Delta^TH_\Delta\Delta=0$, as the form of Hard CLF-CBF QP.
\end{proof}
\begin{table}[htbp]
\caption{Comparisons of Slack Variables in Existing QP-based methods}

\label{tab:diffslackcons}
\centering
\begin{tabular}{ccccccc}
\toprule
\textbf{Methods}&$\mathcal{U}_\Delta$&$\mathcal{V}_\varepsilon$\\
\midrule
Hard CLF-CBF QP&$\{0\}\times\{0\}$&$\{0\}\times\{0\}$\\
\cline{1-1}
CLF-CBF QP&$\mathbb{R}\times\{0\}$&$\mathbb{R}\times\{0\}$\\
\cline{1-1}
{Optimal-decay} &$h(x)\neq 0:\mathbb{R}^2$& $h(x)\neq 0:\mathbb{R}^2$ \\{CLF-CBF QP}&$h(x)= 0:\mathbb{R}\times\{0\}$&$h(x)= 0:\mathbb{R}\times\{0\}$\\
\bottomrule
\end{tabular}
\end{table}
From \autoref{lemma1} and its proof, we can find that CLF-CBF QP, Optimal-decay CLF-CBF QP and Hard CLF-CBF QP are certain special cases of Unified CLF-CBF QP with $\mathcal{U}_\Delta$ of different values, as listed in details in \autoref{tab:diffslackcons}. In summary,  \textbf{CLF-CBF QP is equal to adding a slack variable in the CLF constraint on the basis of Hard CLF-CBF QP, and Optimal-decay CLF-CBF QP is  equal to adding another slack variable (as long as \boldmath ${h(x)\neq 0}$) on the basis of CLF-CBF QP to relax CBF constraints.}

Furthermore, a standardized form for Unified CLF-CBF QP in \eqref{unified_func}-\eqref{slackcons} is presented in  \autoref{lemma2}.

\begin{lemma}\label{lemma2}
Given the control system \eqref{model0}-\eqref{inputcons} with a nominal controller $k(x)$, if choose $S^TS=H$ and $S_\Delta=\sqrt{H_\Delta}=\text{diag}\left[\begin{array}{cc}
      p_1&p_2\end{array}\right]$, then
Unified CLF-CBF QP in \eqref{unified_func}-\eqref{slackcons} is equivalent to the following standardized form (called \textbf{SU CLF-CBF QP} for abbreviation):
\begin{subequations}
    \begin{align}
        \label{euddistance}[v^*,\varepsilon^*]&=\mathop{\arg\min}\limits_{[v,\varepsilon]}\frac{1}{2}(v^Tv+\varepsilon^T\varepsilon)\notag\\
        &=\mathop{\arg\min}\limits_{[v,\varepsilon]}\frac{1}{2}
        \left[\begin{array}{cc}
             v  \\
             \varepsilon 
        \end{array}\right]^T
        \left[\begin{array}{cc}
             v  \\
             \varepsilon 
        \end{array}\right]
    \end{align}
    \begin{align}
    \textbf{s.t.}\label{final_CLFcons}\quad\quad&L_{f^*} V(x)+L_{g^*} V(x) v+\lambda V(x) \leq \frac{\varepsilon_1}{p_1}\\
        \label{final_CBFcons}&L_{f^*} h(x)+L_{g^*} h(x) v+\gamma h(x) \geq \frac{\varepsilon_2}{p_2}\\
        \label{final_inputcons} &v\in \mathcal{V}\\
        \label{final_slackcons}&\varepsilon\in \mathcal{V}_\varepsilon
    \end{align}
\end{subequations}
where $v=S(u-k(x))$, $\varepsilon=\left[\begin{array}{cc}
      \varepsilon_1\\
     \varepsilon_2
\end{array}\right]=S_\Delta\Delta$, $L_{f^*}(\cdot)=L_f(\cdot)+L_g(\cdot)k(x)$, $L_{g^*}(\cdot)=L_g(\cdot)S^{-1}$, and $\mathcal{V}$ is a convex polytope including the origin if $U$ is a convex polytope including $k(x)$.
\end{lemma}
\begin{proof}
 Firstly, since $H$ is a positive symmetric matrix and $H_\Delta$ is diagonal, we can set $S^TS=H$, $S_\Delta=\sqrt{H_\Delta}$, $v=S(u-k(x))$ and $\varepsilon=S_\Delta\Delta$ where $S$ is also a positive symmetric matrix and $S_\Delta$ is still diagonal, then the objective function \eqref{unified_func} turns into \eqref{euddistance} obviously;
 
 Secondly, the system \eqref{model0}, input constraints \eqref{inputcons0} and slack constraints $\mathcal{U}_\Delta$ are converted to
\begin{align}
    \dot{x}&=f(x)+g(x)(S^{-1}v+k(x))\notag\\
    &=\mathop{\underline{f(x)+g(x)k(x)}}\limits_{\triangleq f^*(x)}+\mathop{\underline{g(x)S^{-1}}}\limits_{\triangleq g^*(x)}v\notag\\
    \label{keepaffine}&=f^*(x)+g^*(x)v\\
    \label{keepconvex}u&\in U\Rightarrow v\in \mathcal{V}=\{ v|v=S(u-k(x)), u\in U\} \\
    \label{slackunified}\Delta&\in \mathcal{U}_\Delta\Rightarrow \varepsilon\in \mathcal{V_\varepsilon}=\{ \varepsilon|\varepsilon=S_\Delta\Delta, \Delta\in \mathcal{U}_\Delta\} 
    \end{align}
which shows the system \eqref{keepaffine} is control-affine with input $v$, and $\mathcal{V}_\varepsilon$ is still a subset of $\mathbb{R}^2$, so \eqref{inputcons3},\eqref{slackcons} are converted into \eqref{final_inputcons} and \eqref{final_slackcons}, respectively;

As a result, for the system \eqref{keepaffine} subject to \eqref{keepconvex}, we have
\begin{equation}
\begin{aligned}
    L_{f^*}(\cdot)&=\nabla^T(\cdot)f^*(x)\\
    &=\nabla^T(\cdot)(f(x)+g(x)k(x))\\
    &=\nabla^T(\cdot)f(x)+\nabla^T(\cdot)g(x)k(x)\\
    &=L_f(\cdot)+L_g(\cdot)k(x)\notag
\end{aligned}
\end{equation}
and
\begin{equation}
\begin{aligned}
    L_{g^*}(\cdot)&=\nabla^T(\cdot)g^*(x)\\
    &=\nabla^T(\cdot)g(x)S^{-1}\\
    &=L_g(\cdot)S^{-1}\notag
\end{aligned}
\end{equation}
from \autoref{Def_CLF} and \autoref{Def_CBF}, which shows \eqref{eq:CLFcons3}-\eqref{CBFcons3} turn into \eqref{final_CLFcons}-\eqref{final_CBFcons};

Finally, let's examine the new input constraint of $v$.
Since $f:u\mapsto v=S(u-k(x))$ is an affine transformation, and $U$ is convex, then $\mathcal{V}$ is convex;

And since $U$ is a convex polytope, i.e.,
    \begin{equation}
        U:\{u|Au\preceq b,A\in \mathbb{R}^{p\times m},b\in \mathbb{R}^p\}\notag
    \end{equation}
    then $\mathcal{V}$ can be written as
    \begin{equation}
        \mathcal{V}:\{v|AS^{-1}v\preceq b-Ak(x),A\in \mathbb{R}^{p\times m},b\in \mathbb{R}^p\}\notag
    \end{equation}
 which is a convex polytope as well. 
    
    In addition, 
    \begin{equation}\notag
    \begin{gathered}
        k(x)\in U\Rightarrow Ak(x)\preceq b\\
        \Rightarrow AS^{-1}0=0\preceq b-Ak(x)\\
        \Rightarrow v=0 \in \mathcal{V}
    \end{gathered}
    \end{equation}
    which shows $\mathcal{V}$ includes the origin. 

To sum up, the optimal solution of SU CLF-CBF QP in \eqref{euddistance}
-\eqref{final_slackcons} is equivalent to the solution of Unified CLF-CBF QP in \eqref{unified_func}-\eqref{slackcons}, i.e.,
$v^*=S(u^*-k(x))$.
\end{proof}
    
\begin{corollary}\label{cor:invarsubspace}
    If $\mathcal{U}_\Delta=A\times B$ where $A,B\in \{\mathbb{R},\{0\}\}$, then $\mathcal{U}_\Delta$ is an invariant subspace of the invertible mapping $S_\Delta:\mathcal{U}_\Delta\rightarrow\mathcal{V}_\varepsilon$, i.e., $\mathcal{U}_\Delta=\mathcal{V}_\varepsilon$.
\end{corollary}

The succinct form of SU CLF-CBF QP implies a straightforward geometrical explanation of the objective function in \eqref{euddistance}: the Euclidean distance from the extended input vector $\left[v  \ \varepsilon \right]^T$ to the origin. The corresponding slack variables of existing QP-based methods in SU CLF-CBF QP form are shown as \autoref{tab:diffslackcons}. 

\subsection{Limit-weight form of Unified CLF-CBF QP}\label{subsec:limit-weight}
Our proposed Safety-first CLF-CBF QP treats safety, stability and system performance at different levels of priorities. On the contrary, in SU CLF-CBF QP all different objectives are coined into one objective function with different weighting coefficients. In this subsection, we build a connection between our proposed Safety-first CLF-CBF QP and the SU CLF-CBF. That is, the former one is a limit-weight form of Unified CLF-CBF QP.
\begin{theorem}\label{theorem:limitweight}
   The control sequences $u^*$ and slack variables $\delta_1^*,\delta_2^*$ obtained by solving  Safety-first CLF-CBF QP \eqref {op:safetyfirst_1}-\eqref {op:safetyfirst_3} are also the optimal solution of the following limit-weight Unified CLF-CBF QP
    \begin{equation}\label{op:limitweight}
        \begin{aligned}
            &\min \quad (u-k(x))^TH(\cdot)+q\delta_1^2+q^2\delta_2^2\\
            \textbf{s.t.}\quad\quad
            &L_{f}V(x)+L_{g}V(x)u+\lambda V(x)\leq \delta_1\\
            &L_{f}h(x)+L_{g}h(x)u+\gamma h(x)\geq \delta_2\\
            &u\in \mathcal{U}
        \end{aligned}
    \end{equation}
when the limit weight $q\rightarrow +\infty$.
\end{theorem}
\begin{figure*}[th]
\centering
\subfigure[Illustration of the different regions defined by the CBF constraint, CLF constraint and input constraints of SU CLF-CBF QP.]{
\vspace{-0.4cm}
\label{fig1a}
\includegraphics[scale=0.45]{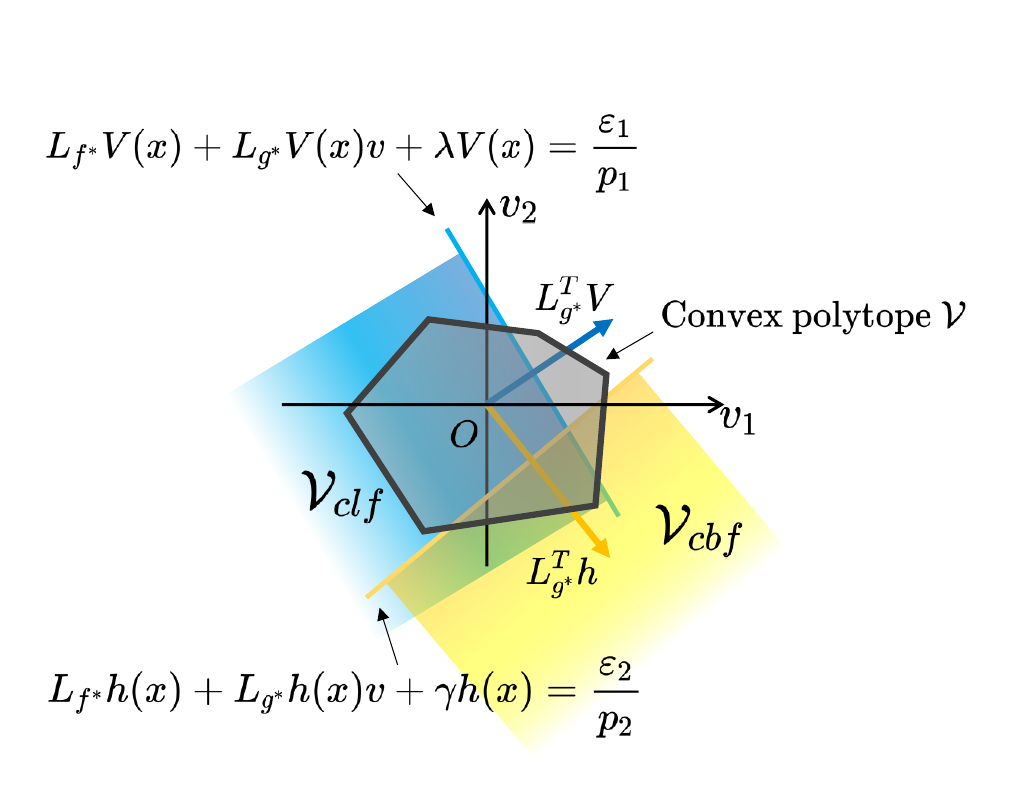}}
\subfigure[The case $\mathcal{V_{\varepsilon}}=\mathbb{R}\times \{0\}$ and the corresponding CLF-CBF QP is feasible.]{
\vspace{-0.4cm}
\label{fig1b}
\includegraphics[scale=0.45]{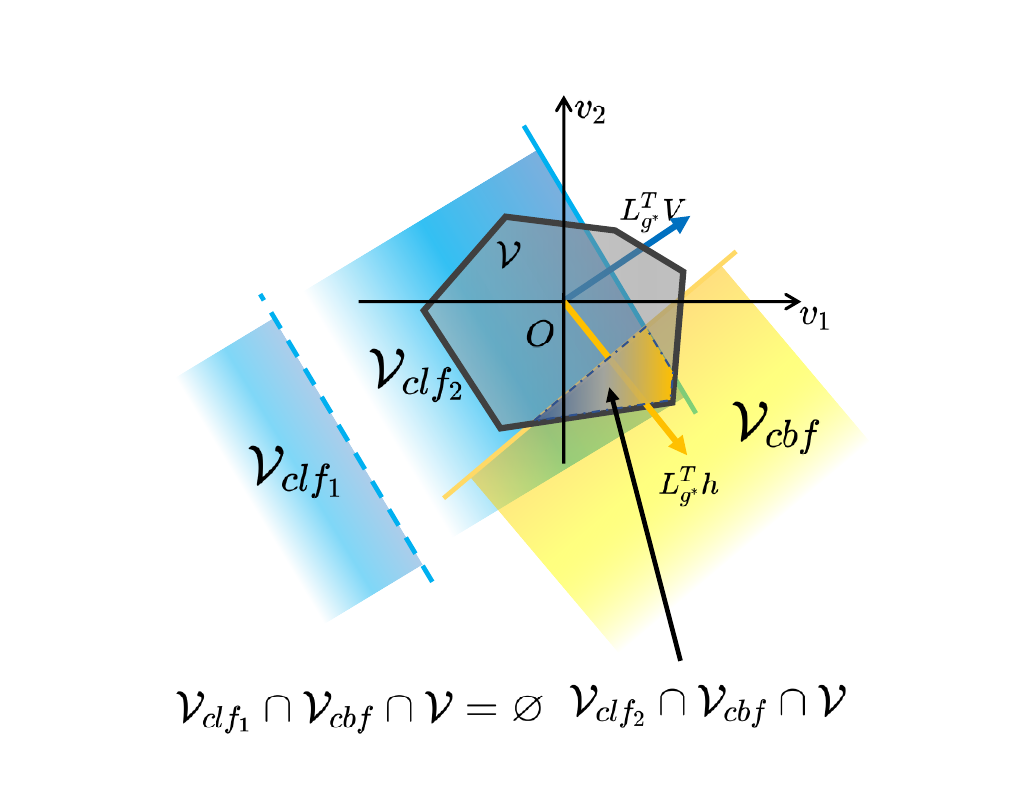}}
\subfigure[The case that $\mathcal{V_{\varepsilon}}=\mathbb{R}\times \{0\}$ and the corresponding CLF-CBF QP is infeasible.]{
\vspace{-0.2cm}
\label{fig1c}
\includegraphics[scale=0.45]{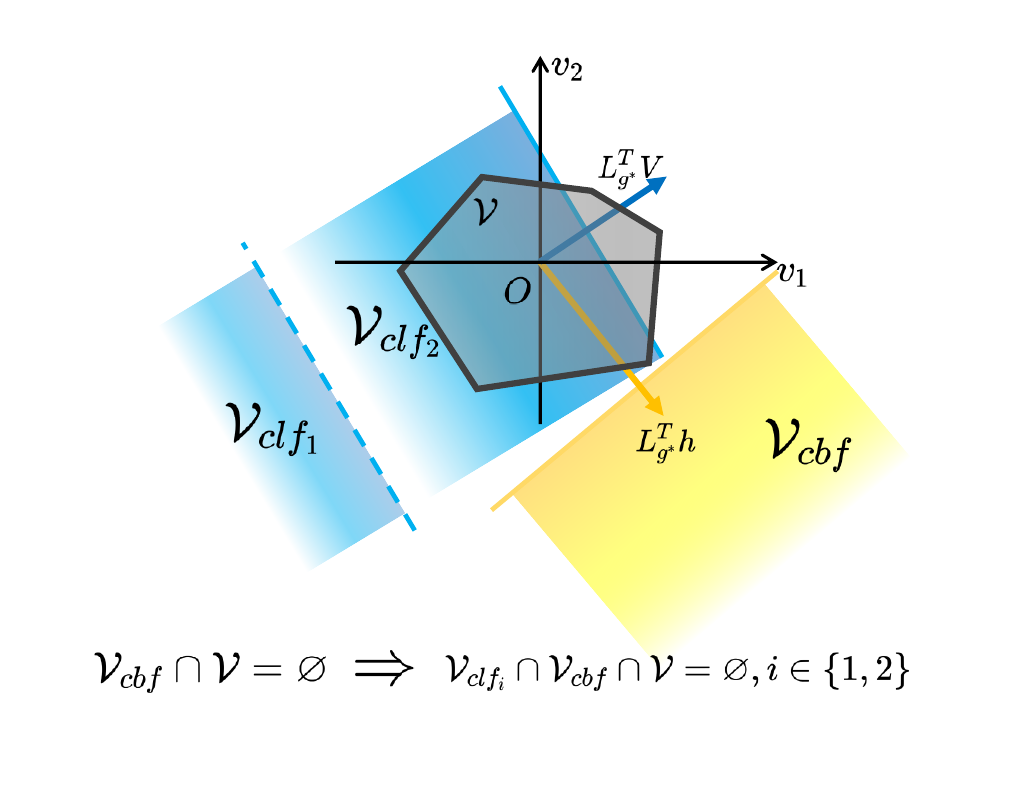}}
\subfigure[The case that $\mathcal{V_{\varepsilon}}=\mathbb{R}\times \mathbb{R}$ and the corresponding Optimal CLF-CBF QP ($h(x)\neq 0$) and Safety-first CLF-CBF QP are feasible.]{
\vspace{-0.4cm}
\label{fig1d}
\includegraphics[scale=0.45]{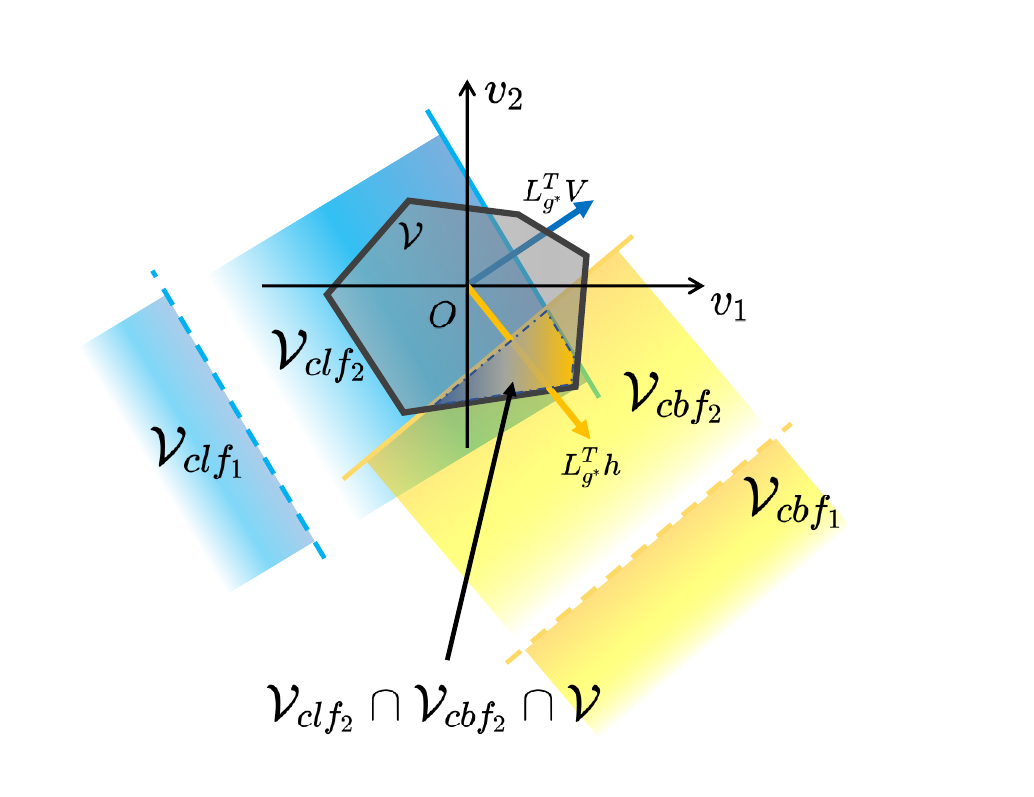}}
\caption{\centering{
Illustrations of feasibility and infeasiblity of SU CLF-CBF QP.}}
\label{figs:1a-1d}
\end{figure*}

\begin{proof}
    First, we prove that $[u^*,\delta_1^*,\delta_2^*]$ satisfy all constraints of \eqref{op:limitweight}. Since $u^*$ is the solution of \eqref{op:safetyfirst_3}, we have that 
    \begin{align*}
        &L_{f} V(x)+L_{g} V(x) u^*+\lambda V(x) \leq \delta_1^*\\
        &L_{f} h(x)+L_{g} h(x) u^*+\gamma h(x) \geq \delta_2^*\\
        &u^*\in\mathcal{U}
    \end{align*}

    Then, we prove that $[u^*,\delta_1^*,\delta_2^*]$ is optimal. Assume the optimal solution of \eqref{op:limitweight} is $[u',\delta_1',\delta_2']$.  Minimizing the following objective function is equivalent to that in \eqref{op:limitweight},
    \begin{equation*}
        \frac{1}{q^2}(u-k(x))^TH(\cdot)+\frac{1}{q}\delta_1^2+\delta_2^2
    \end{equation*}
    then
    \begin{equation*}
         \lim_{q \to +\infty}\frac{1}{q^2}(u-k(x))^TH(\cdot)+\frac{1}{q}\delta_1^2+\delta_2^2 = \delta_2^2
    \end{equation*}
    which shows the objective function is equivalent to \eqref{op:safetyfirst_1}, so that $\delta_2'=\delta_2^*$. Since $\delta_2$ is determined, then \eqref{op:limitweight} can be rewritten as
    \begin{equation*}
        \begin{aligned}
            &\min \quad (u-k(x))^TH(\cdot)+q\delta_1^2\\
            \textbf{s.t.}\quad\quad
            &L_{f}V(x)+L_{g}V(x)u+\lambda V(x)\leq \delta_1\\
            &L_{f}h(x)+L_{g}h(x)u+\gamma h(x)\geq \delta_2^*\\
            &u\in \mathcal{U}
        \end{aligned}
    \end{equation*}
    and its objective function can also be converted into 
    \begin{equation*}
        \frac{1}{q}(u-k(x))^TH(\cdot)+\delta_1^2
    \end{equation*}
    Similarly, we have
    \begin{equation*}
         \lim_{q \to +\infty}\frac{1}{q}(u-k(x))^TH(\cdot)+\delta_1^2 = \delta_1^2
    \end{equation*}
    which infers $\delta_1'=\delta_1^*$. Simliarly to the above procedure, we can further get $u'=u^*$.
\end{proof}

\autoref{theorem:limitweight} indicts that by setting a great enough $q$ in Unified CLF-CBF QP, we can obtain an approximate solution to Safety-first CLF-CBF QP in one step.

\section{Analysis of QP Based Methods: Feasibility, Safety and Convergence}\label{sec:shortcomings}

In this section, we will systematically investigate  the three existing CLF-CBF QP methods summarized in \autoref{subsec: Preliminaries} and compare them with our proposed method from three dimensions: feasibility, safety and convergence degradation.  

\subsection{Feasibility Analysis}\label{subsec:feasiblity}
Considering the feasible region in input space defined by \eqref{final_CLFcons}-\eqref{final_inputcons}, both CLF constraint \eqref{final_CLFcons} and CBF constraint \eqref{final_CBFcons} are linear on the variable $v$, which means each inequation determines a half-space in $\mathcal{V}$, and the boundary is a hyperplane with normal vector $L_{g^*}^{T}(\cdot)$, as shown in \autoref{fig1a}. 

To simplify mathematical expressions, we define
\begin{equation}
    \begin{aligned}
    \mathcal{V}_{clf}(x)&=\{v|L_{f^*} V(x)+L_{g^*} V(x) v+\lambda V(x) \leq \frac{\varepsilon_1}{p_1}\}\\
    \mathcal{V}_{cbf}(x)&=\{v|L_{f^*} h(x)+L_{g^*} h(x) v+\gamma h(x) \geq \frac{\varepsilon_2}{p_2}\}
    \end{aligned}
\end{equation}
to represent the half-spaces determined by \eqref{final_CLFcons} and \eqref{final_CBFcons}, respectively. Accordingly, we have the following result on feasibility.


\begin{lemma}\label{theorem:feasibility}
Given that $\mathcal{V}$ is a closed convex polytope including the origin,  if $\mathcal{V}_\varepsilon=\mathbb{R}^2$, then SU CLF-CBF QP \eqref{euddistance}-\eqref{final_slackcons} is always feasible; otherwise, the feasibility cannot be guaranteed all the time.
\end{lemma}

\begin{proof}
If $\mathcal{V}_\varepsilon=\mathbb{R}^2$, choose $\varepsilon_{v}=p_1[L_{f^*} V(x)+L_{g^*} V(x) v_0+\lambda V(x)],\varepsilon_{h}=p_2[L_{f^*} h(x)+L_{g^*} h(x) v_0+\gamma h(x)]$, then $\forall v_0\in\mathcal{V}$, $v_0\in \mathcal{V}_{clf}\cap\mathcal{V}_{cbf}\cap\mathcal{V}$, which means $\mathcal{V}_{clf}\cap\mathcal{V}_{cbf}\cap\mathcal{V}\neq\varnothing$, thus SU CLF-CBF QP \eqref{euddistance}-\eqref{final_slackcons} is always feasible.

Next, consider the case that $\mathcal{V}_\varepsilon\neq\mathbb{R}^2$. When one of $\varepsilon_1$ and $\varepsilon_2$ is fixed as a constant, e.g., $\varepsilon_2=c$ where $c$ is a constant, then $\mathcal{V}_{cbf}$ is fixed. Define
\begin{align*}
r=&\mathop{\arg\max}\limits_{v}L_{g^*}h(x)v\notag\\
    \textbf{s.t.}\quad\quad&v\in \mathcal{V}
\end{align*}
then $r$ means the the maximum projection along the direction of $L_{g^*}^Th(x)$ in $\mathcal{V}$. Similarly, let 
\begin{align*}
    r_0=&\mathop{\arg\min}\limits_v \Vert v \Vert\notag\\
    \textbf{s.t.}\quad\quad&v\in \mathcal{V}_{cbf}
\end{align*}
where $ r_0 $ means the minimum distance from $\mathcal{V}_{cbf}$ to origin. 

If $ \Vert r \Vert  \geq  \Vert r_0 \Vert$, as shown as \autoref{fig1b}, $\mathcal{V}_{cbf}\cap\mathcal{V}\neq\varnothing$. Choose any $v_0\in\mathcal{V}_{cbf}\cap\mathcal{V}$, and let $\varepsilon_{v}=p_1[L_{f^*} V(x)+L_{g^*} V(x) v_0+\lambda V(x)]$, we can easily find that $v_0\in\mathcal{V}_{clf}\cap\mathcal{V}_{cbf}\cap\mathcal{V}$, so $\mathcal{V}_{clf}\cap\mathcal{V}_{cbf}\cap\mathcal{V}\neq\varnothing$. In this case, the optimization problem is feasible.

If $\Vert r\Vert < \Vert r_0 \Vert $, as shown as \autoref{fig1c}, $\mathcal{V}_{cbf}\cap\mathcal{V}=\varnothing$, so in this case the optimization problem is infeasible.
\end{proof}

Referring to \autoref{lemma1} and \autoref{lemma2}, SU CLF-CBF QP can be reduced to CLF-CBF QP \eqref{QP_model}-\eqref{inputcons} and Optimal-decay CLF-CBF QP \eqref{opdecay_objective}-\eqref{inputcons2} with proper mathematical transformation, respectively. Then a corollary follows \autoref{theorem:feasibility} as below:
\begin{corollary}\label{conclusion_feas}
Consider  CLF-CBF QP \eqref{QP_model}-\eqref{inputcons}, Optimal-decay CLF-CBF QP \eqref{opdecay_objective}-\eqref{inputcons2} and Safety-first CLF-CBF QP \eqref{op:safetyfirst_1}-\eqref{op:safetyfirst_3}, given that the control constraint $\mathcal{U}$ is a closed convex polytope including the nominal controller $k(x)$, then
\begin{enumerate}
 \item If $\gamma$ in \eqref{CBFcons} is not chosen properly, then there may not exist $\delta$ such that CLF-CBF QP \eqref{QP_model}-\eqref{inputcons} is feasible;
 \item If $h(x)\neq 0$, then there always exists $\delta, \omega$ such that Optimal-decay CLF-CBF QP \eqref{opdecay_objective}-\eqref{inputcons2} is feasible;
 \item Safety-first CLF-CBF QP \eqref{op:safetyfirst_1}-\eqref{op:safetyfirst_3} is feasible at any state $x$.
\end{enumerate}
\end{corollary} 
\autoref{fig1b}-\autoref{fig1d} show the typical situations might encountered when using SU CLF-CBF QP. 

\subsection{Safety Analysis}\label{subsec:safety}
As proved in \cite{QP_TAC,theoryandapp_ECC,PlanarQuad_ACC}, if CBF constraint is hard, i.e., $\varepsilon_2=0$ in SU CLF-CBF QP, or $h(x)=0$ in CLF-CBF QP and Optimal-decay CLF-CBF QP, then the solutions of all these QPs guarantee safety of the system. However, it is worth pointing out that the feasibility of QPs, i.e., the existence of control solutions at every time instant, is the prerequisite for safety guarantee. Hence, referring to \autoref{conclusion_feas}, it is known that \textbf{\textit{only Optimal-decay CLF-CBF QP \eqref{opdecay_objective}-\eqref{inputcons2} guarantees the safety when \boldmath${x_0\in \operatorname{Int}\mathcal{(C)}}$}}. 
Referring to  \autoref{def:subsafety}, the  sub-safety of SU CLF-CBF QP is stated below:

\begin{figure}[hbt]
    \centering
        \centering
    \includegraphics[scale=0.7]{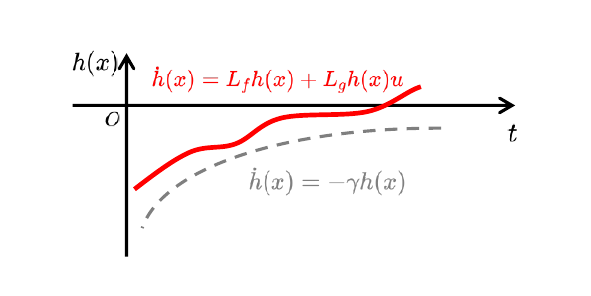}
    \caption{The bound determined by $\gamma$ when $h(x)<0$ at $t=0$.}
    \label{fig4}
\end{figure}

\begin{lemma}\label{theorem:subsafety}
    Assume SU CLF-CBF QP in \eqref{euddistance}-\eqref{final_inputcons} is feasible, then the system is sub-safe if
    \begin{equation}\label{eq_subsafe}
        \varepsilon_2^*>p_2\gamma h(x)
    \end{equation}
    is always satisfied when $h(x)<0$.
\end{lemma}
\begin{proof}
    If \eqref{eq_subsafe} is satisfied, it follows from \eqref{final_CBFcons} 
    \begin{align*}
        &L_{f^*} h(x)+L_{g^*} h(x) v+\gamma h(x) \geq \frac{\varepsilon_2^*}{p_2}\\
        \Rightarrow\quad&L_{f^*} h(x)+L_{g^*} h(x) v \geq \frac{\varepsilon_2^*}{p_2}-\gamma h(x)> 0\\
        \Rightarrow\quad&\dot{h}>0
    \end{align*}

    Assume $h(x)<0$ at $ t_0 $, then $\dot{h}>0$ and it has upper bound $0$, so there exists $h_c$ is the limit of $h(x)$, i.e., $h(x)\rightarrow h_c$ as $ t \rightarrow \infty $. However, since $\dot{h}>0$, there exists a neighborhood satisfies $h>h_c$, which conflicts with $h(x)\rightarrow h_c$ if $h_c<0$. So $h_c=0$, which means the system is sub-safe.
\end{proof}

\begin{corollary}
    CLF-CBF QP \eqref{QP_model}-\eqref{inputcons} always guarantees sub-safety of the system if it is feasible, but Optimal-decay CLF-CBF QP \eqref{opdecay_objective}-\eqref{inputcons2} does not; For Safety-first CLF-CBF QP, if the optimal solution to sub-problem 1 \eqref{op:safetyfirst_1} is $\varepsilon^*_2=0$, then it always guarantees sub-safety of the system, otherwise, it violates the sub-safety with a minimal relaxation of the CBF constraints by the slack variable $\varepsilon^*_2$.
\end{corollary}

\begin{proof}
When $h(x)<0$, $\gamma h(x)<0$ holds for $\forall \gamma>0$, so if in CLF-CBF QP, i.e., the case of $\varepsilon_2=0$ in SU CLF-CBF QP, \eqref{eq_subsafe} is always satisfied, and there is also a lower bound of $h$, shown as \autoref{fig4}. 

In Optimal-decay CLF-CBF QP,  $\varepsilon_2$ can vary from $-\infty$ to $+\infty$ and hence it is possible to violate \eqref{eq_subsafe}, losing guarantees of sub-safety.

In Safety-first CLF-CBF QP, if $\varepsilon_2=0$ is feasible, then $\varepsilon_2^*=0$, which guarantees the sub-safety. if $\varepsilon_2=0$ is infeasible, then $\varepsilon_2^*$ is the minimal-slack of CBF constraint, guarantees the sub-safety as much as possible.
\end{proof}

 
To summarize, our Safety-first CLF-CBF QP is always feasible and guarantee safety and sub-safety as much as possible, which cannot achieve by exiting  methods.

\begin{table*}[hbtp]
\caption{Comparisons of different methods in terms of feasibility, safety/sub-safety and convergence}
\centering
\label{tab:comparison 2}
\begin{tabular}{ccccccc}
\toprule
\multirow{2}{*}{\textbf{Methods}} &\multicolumn{2}{c}{\textbf{Feasibility}} &\multicolumn{2}{c}{\textbf{Safety/Sub-safety}} &\multirow{2}{*}{\textbf{Convergence}}\\
&$h(x)\neq0$&$h(x)=0$&$h(x)\geq0$&$h(x)<0$\\
\midrule
Hard CLF-CBF QP&&&\checkmark&\checkmark&\checkmark\\ 
CLF-CBF QP&&&\checkmark&\checkmark& \\
Optimal-decay CLF-CBF QP&\checkmark&&\checkmark&\\
\textbf{Safety-first CLF-CBF QP}&\checkmark&\checkmark&\checkmark&\checkmark&\checkmark \\
\bottomrule
\end{tabular}
\end{table*}

\subsection{Convergence degradation}\label{subsec:convergence}
The hyper-parameters $\gamma$ in the CBF condition \eqref{CBFperfect} and  $\lambda$ in the CLF condition \eqref{CLFperfect} specify exponential functions with decay rate  $\gamma$ and $\lambda$ as the boundaries, as shown in \autoref{fig5a} and \autoref{fig5b}.  However, the introduced slack variables, such as $\delta_1$ and $\delta_2$ in Unified CLF-CBF QP \label{CLFcons3}-\eqref{CBFcons3}, which aims to solving potential infeasibility could change the decay rates different from pre-specified $\gamma$ or $\lambda$. We call such a phenomenon \textit{\textbf{convergence degradation}}. In this subsection, we will answer the following two questions: \begin{enumerate}
    \item What is the reason for convergence degradation in existing CLF-CBF QP methods?
    \item How will the proposed Safety-First CLF-CBF QP method perform in terms of convergence degradation?
\end{enumerate}

\begin{figure}
    \centering
    \subfigure[]{
    \vspace{-0.4cm}
    \label{fig5a}
    \includegraphics[scale=0.65]{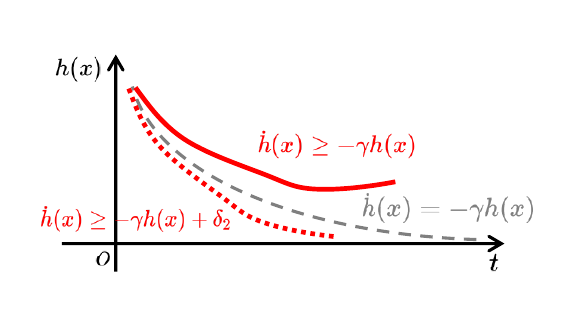}}
    \subfigure[]{
    \vspace{-0.4cm}
    \label{fig5b}
    \includegraphics[scale=0.65]{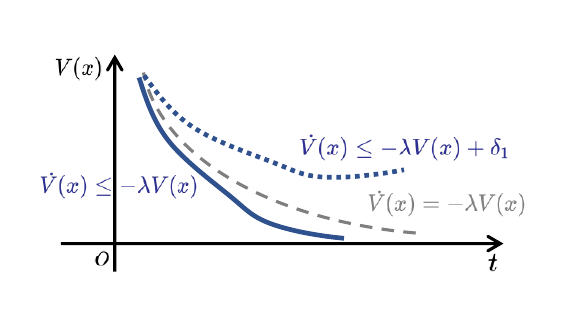}}
    \label{fig5}
    \caption{Convergence degradation caused by slack variables $\delta_1,\delta_2$.}
\end{figure}

To answer the first question, let us first consider a toy example $$ \min \quad f(x)=x^2, \quad 1\leq x \leq 2.$$
By relaxing constraint as $x- 1\geq \varepsilon$, the original problem can be reformulated as:  
\begin{equation*}
    \begin{aligned}
    \min \quad&f(x,\varepsilon)=x^2+10\varepsilon^2\\
    \textbf{s.t.}\quad
    &x- 1\geq \varepsilon\\
    &x-2\leq 0
    \end{aligned}    
\end{equation*}
where $x$ is the main variable and $\varepsilon$ is a slack variable added to guarantee feasibility. $x=1,\varepsilon=0$ is the no-slacked optimal solution with $f(x,\varepsilon)=1$, but it is not the optimal solution of the problem with slack variable $\varepsilon$ to be optimized, since $x=0.99,\varepsilon=-0.01$ provides a better solution with $f(0.99,-0.01)=0.9811<f(1,0)$. From this example, we can find that the introduced slack variable in the optimization objective function could lead to a shift in the optimal solution at the expense of a non-zero slack variable.

Inspired by the observation, we make the following statement on convergence degradation.
 
\begin{proposition}\label{theorem:degrade}
Given that $\left[\begin{array}{cc}v_2^T  & \varepsilon_*^T\end{array}\right]^T$ is the optimal solution of SU CLF-CBF QP \eqref{euddistance}-\eqref{final_slackcons}, and $\left[\begin{array}{ccc}v^T &\varepsilon^T\end{array}\right]^T$ as a feasible solution of the same problem with $\varepsilon=\left[\begin{array}{cc}0  &0 \end{array}\right]^T$, then usually $v_1\neq v_2$ and $\varepsilon_*\neq \varepsilon$.
\end{proposition} 

\autoref{figs:6a-6b} provides a geometric visualization to explain convergence degradation.
\begin{figure}[htbp]
    \centering
\subfigure[]{
\vspace{-0.4cm}
\label{fig6a}
\includegraphics[scale=0.32]{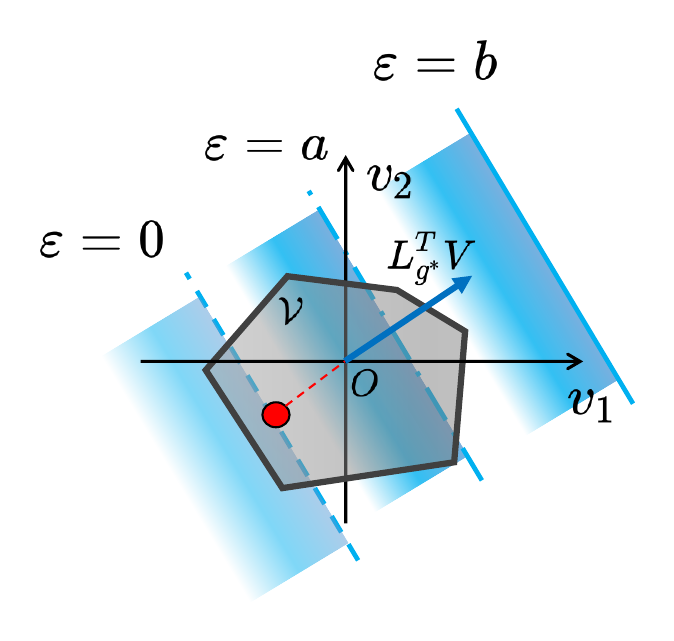}}
\subfigure[]{
\vspace{-0.4cm}
\label{fig6b}
\includegraphics[scale=0.24]{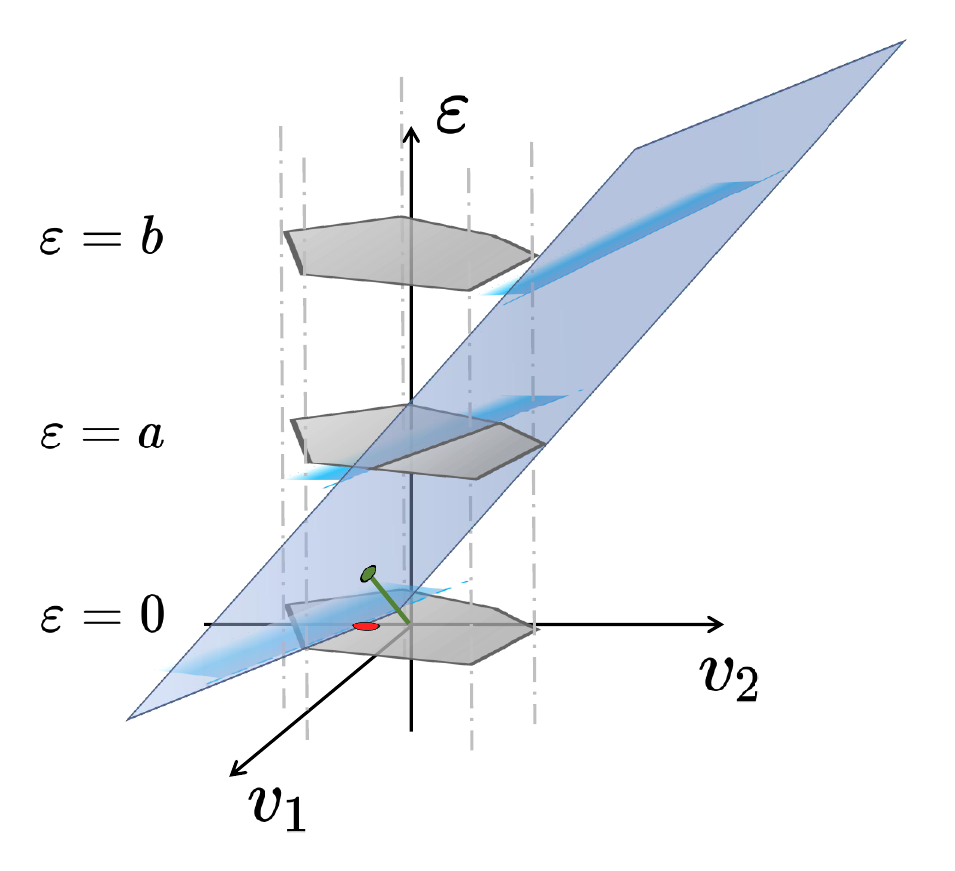}}
\caption{An geometric explanation of convergence degradation. (a) The optimal solution with $\varepsilon=0$ (the red point), and different values of slack variable $\varepsilon$ determine different feasible regions
. (b) the optimal solution when the relaxed variable $\varepsilon$ is also optimized simultaneously. The optimal solution (the green point) is the point with minimal distance from origin, which is usually not identical to the red point.}
\label{figs:6a-6b}
\end{figure}

\begin{figure}
    \centering
\includegraphics[scale=0.25]{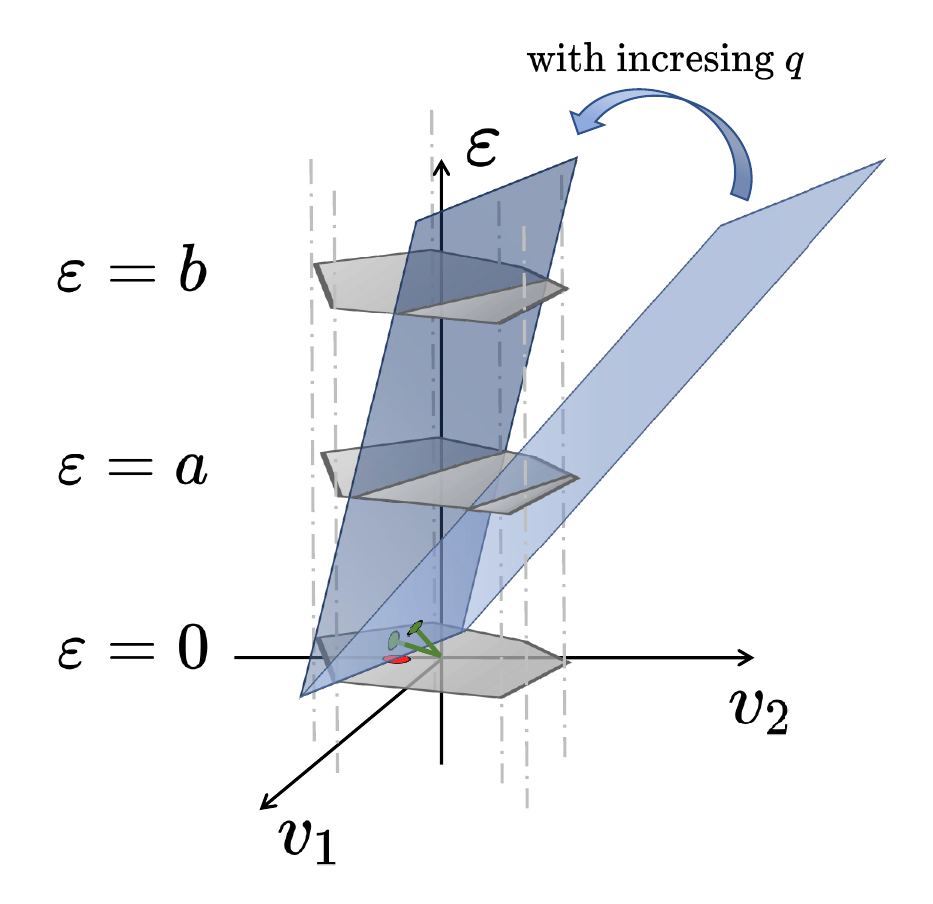}
    \caption{Weight $q$ influenced the slope of boundary plane in Safety-first CLF-CBF QP. A bigger weight $q$ leads the boundary plane (blue) steeper. As $q\to\infty$, the plane is almost vertical, resulting in the yellow point coincide with red point, i.e., the optimal solution is not slacked.}
    \label{fig:7}
\end{figure}

With regard to the second question, our safety-first approach can avoid convergence degradation, and an explanation is given as follows. 


It follows from \autoref{theorem:limitweight} that Safety-first CLF-CBF QP is equivalent to Unified CLF-CBF QP \eqref{op:limitweight} with $q \to +\infty$. From geometric illustration in \autoref{fig:7}, the larger $q$ results in steeper boundary hyperplane. Naturally, when the weight is approaching infinity, the hyperplane is almost "vertical". Referring to \autoref{theorem:degrade}, if $\varepsilon=0$ is a feasible solution, then $\varepsilon^*=0$, i.e., convergence degradation disappears. 

To summarize, our Safety-first CLF-CBF QP ensure expected decay rates when they are achievable. Comparison between our Safety-first CLF-CBF QP and the existing methods are summarized in \autoref{tab:comparison 2}.

\begin{table*}[h]
\caption{settings of four scenarios}
\vspace{-0.4cm}
\label{tab:ACC_restpara}
\begin{center}
\begin{tabular}{lcccccc}
    \toprule
    &Parameters&Case 1 &Case 2&Case 3&Case 4\\
    \midrule
    &$s_0$&$[0,20,100]^T$&$[0,20,100]^T$&$[0,20,20]^T$&$[0,20,20]^T$&\\
    &$v_d$ (m/s)&10&24&10&24\\
    &Remark&Initial safe and $v_d<v_0$&Initial safe but $v_d>v_0$&Initial unsafe but $v_d<v_0$&Initial unsafe and $v_d>v_0$\\
    \bottomrule
    \vspace{-1cm}
\end{tabular}
\normalsize
\end{center}
\end{table*}
\begin{figure*}[h]
    \centering
\subfigure[Case 1]{
\includegraphics[scale=0.14]{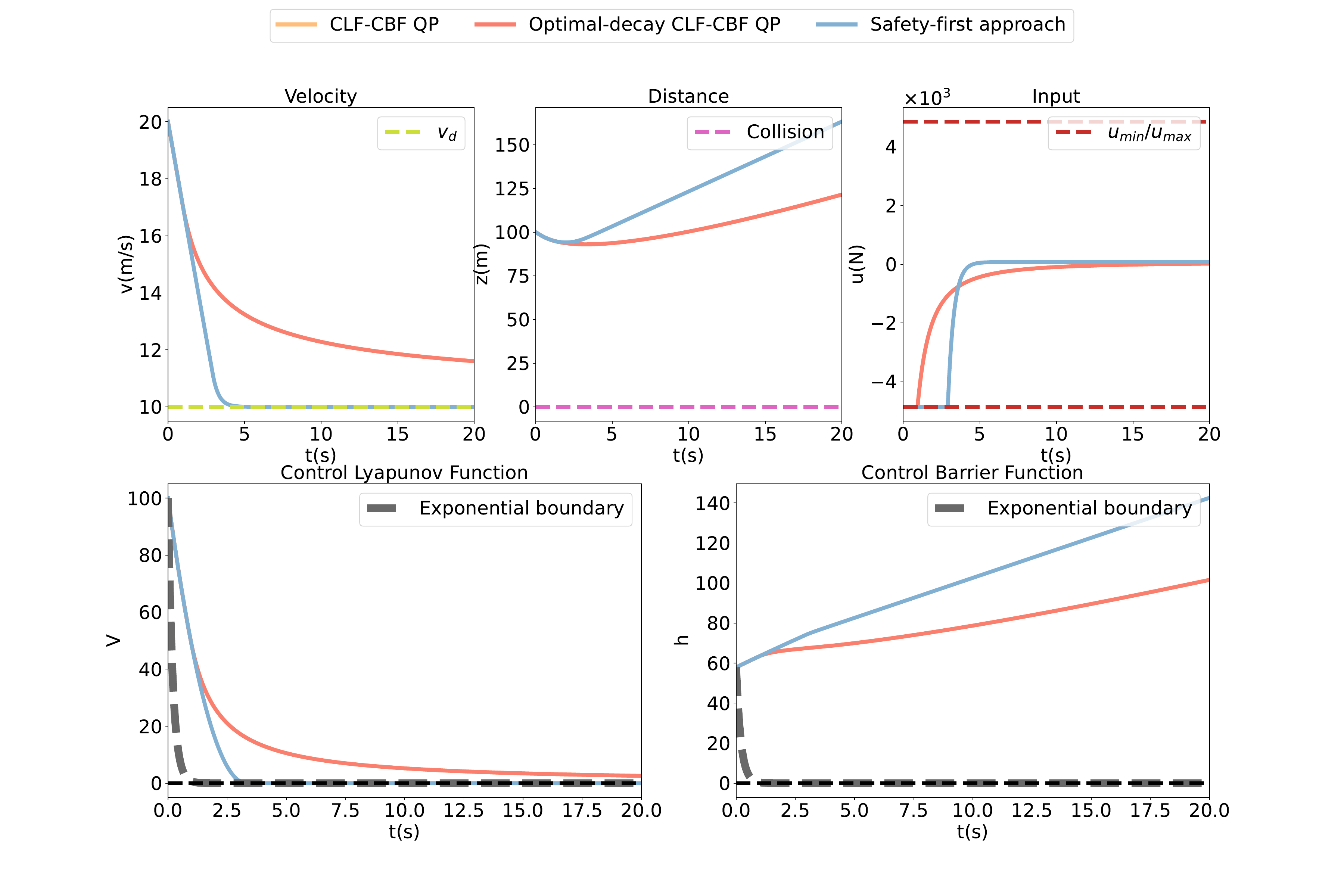}
    \label{fig:ACC_case1}}
\subfigure[Case 2]{
\includegraphics[scale=0.14]{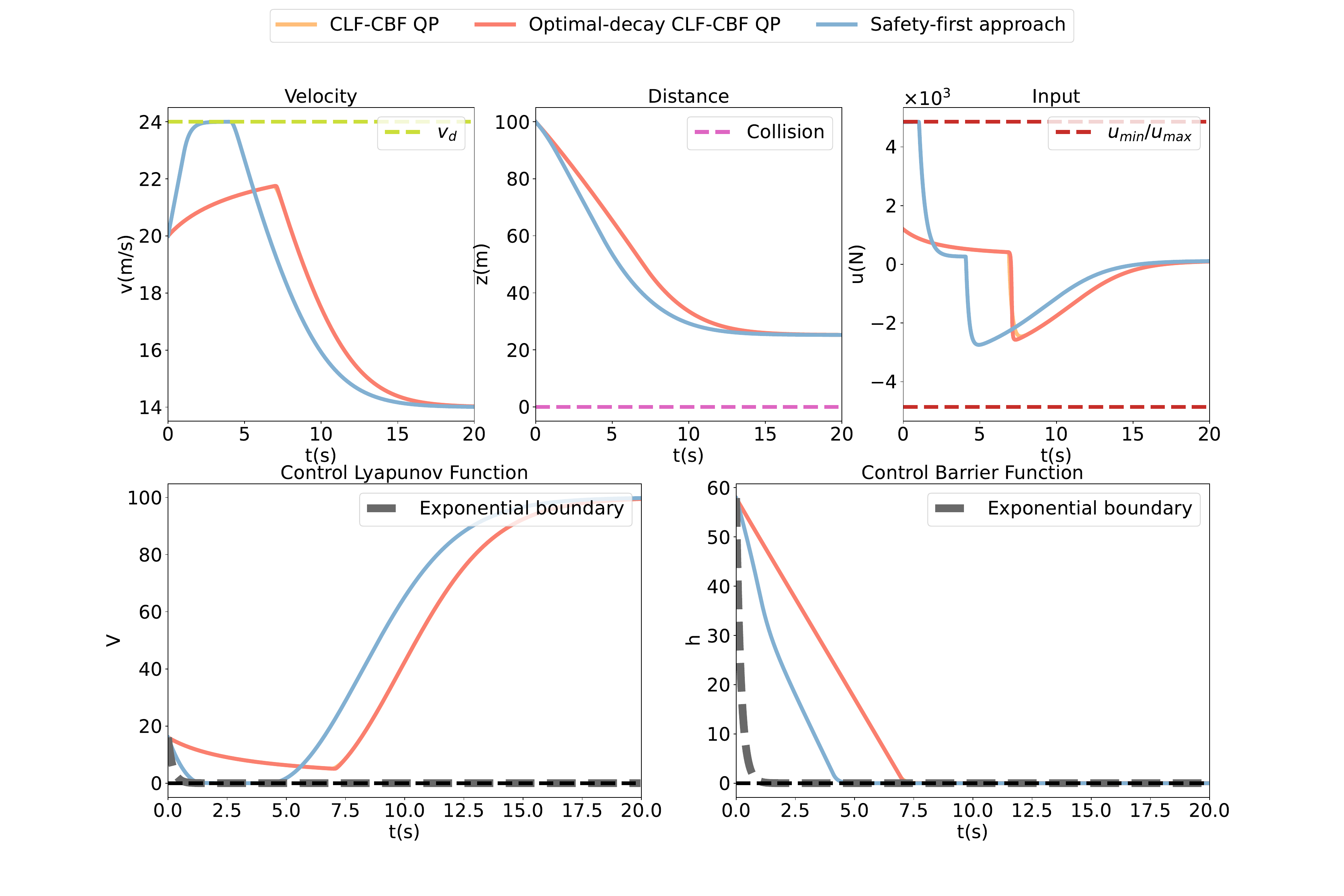}
    \label{fig:ACC_case2}}
\subfigure[Case 3]{
\includegraphics[scale=0.14]{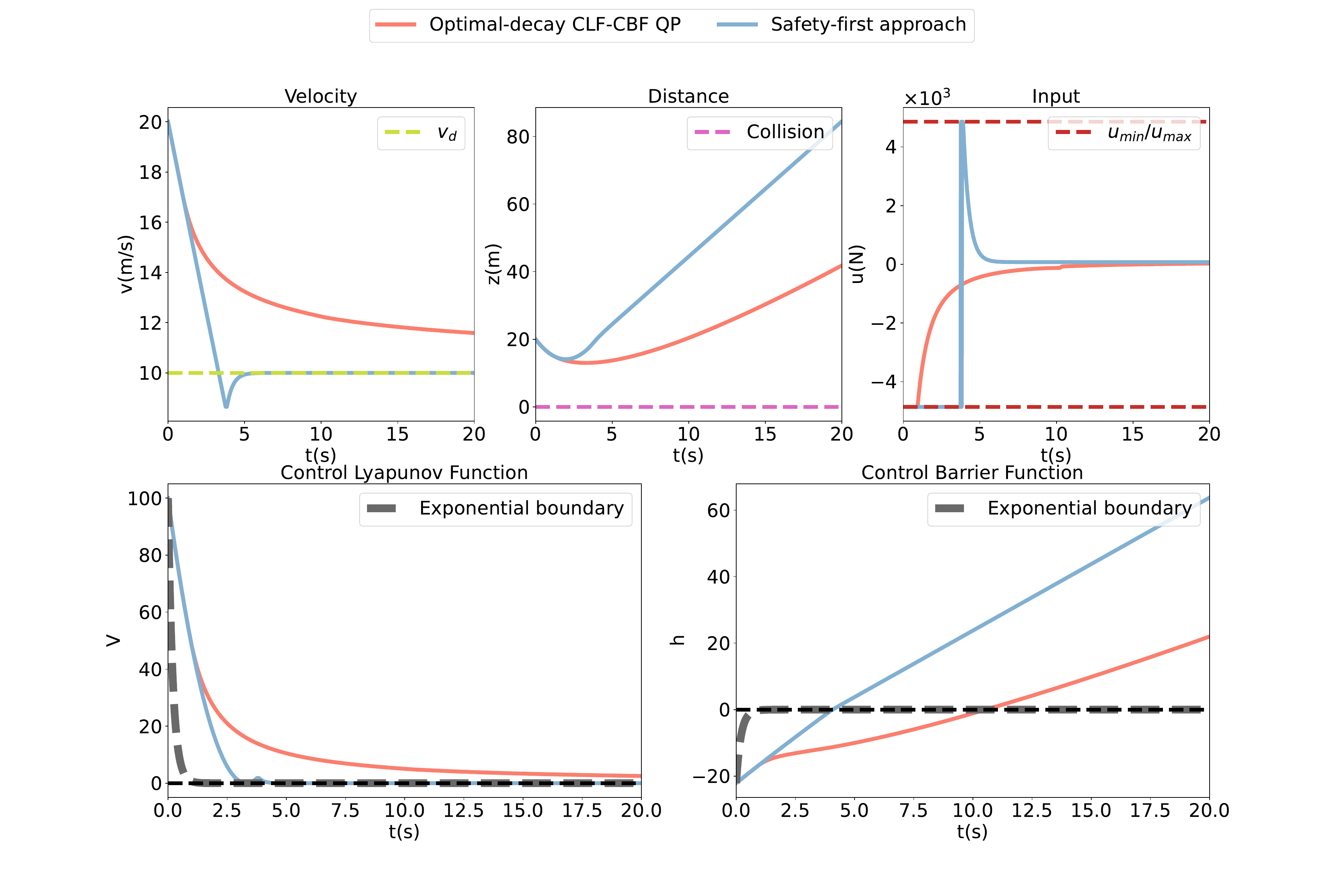}
    \label{fig:ACC_case3}}
\subfigure[Case 4]{
\includegraphics[scale=0.14]{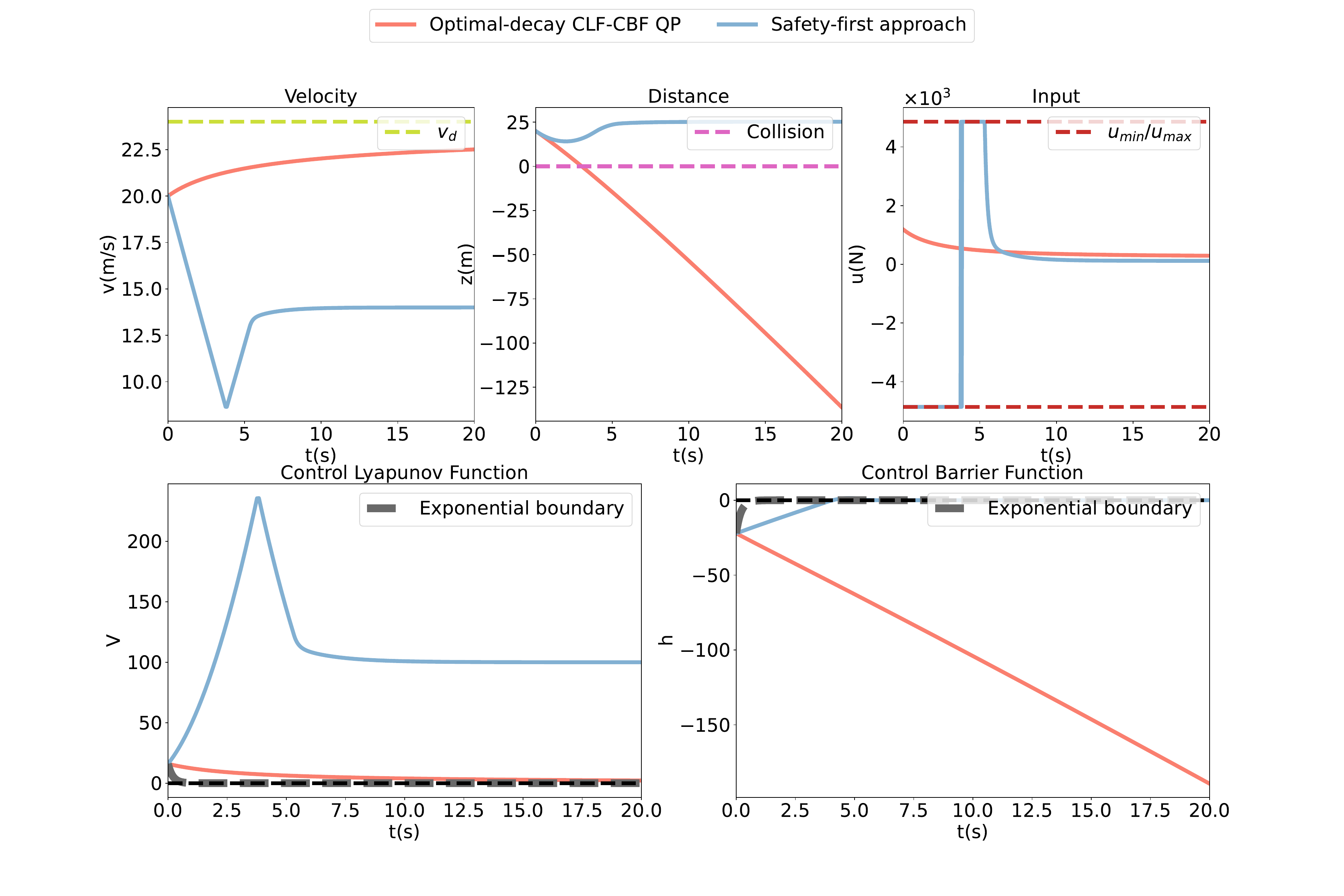}
    \label{fig:ACC_case4}}
    \caption{The results in ACC example. (a)Initial safe and $v_d<v_0$. (b)Initial safe but $v_d>v_0$. (c) Initial unsafe but $v_d<v_0$. (d) Initial unsafe and $v_d>v_0$. }
\label{fig:ACC_cases}
\end{figure*}

\section{Simulation Results}\label{sec:simulation}

To verify the analysis of existing QP-based methods and our proposed safety-first approach, we conduct comparative experiments using two examples \textbf{Adaptive Cruise Control} \cite{acc_CDC} and \textbf{AGV Collision Avoidance} \cite{clf-cbf-helper}.

\subsection{Adaptive Cruise Control}\label{subsec:acc}
\subsubsection{Problem Setup}

Consider the scene that an ego vehicle moves along the one-lane road with a desired velocity $v_d$, and a lead vehicle drive at a speed of $v_0$ in front. Then the system model is formulated as below:
\begin{table}[hbtp]
\caption{Part of parameters in ACC simulations}
\vspace{-0.4cm}
\label{tab:ACC_fixedpara}
\begin{center}
\begin{tabular}{lcccccc}
    \toprule
    &Parameters&Value\\
    \midrule
    &$m$ (kg) &$1650$\\
    &$g$ (m/s$^2$) &$9.81$\\
    &$T_h$ (s)&$1.8$\\
    &$v_0$ (m/s)&14\\
    &$f_0,f_1,f_2$&$0.1,\,5,\,0.25$\\
    &$c_a,c_d$&$0.3,\,0.3$\\
    &$\lambda$&$5$\\
    &$\gamma$&$5$\\
    &$H$&$\frac{2}{1650^2}$\\
    &$p$&$2\times 10^{-3}$\\
    &$\omega_0$&1\\
    &$p_\omega$&0.2\\
    &Control frequency (Hz)&50\\
    \bottomrule
    \vspace{-1cm}
\end{tabular}
\normalsize
\end{center}
\vspace{0.4cm}
\end{table}

\begin{equation}
    \dot{s}=\left[\begin{array}{c}
v \\
-\frac{1}{m} F_{r}(v) \\
v_{0}-v
\end{array}\right]+\left[\begin{array}{c}
0 \\
\frac{1}{m} \\
0
\end{array}\right] u
\end{equation}
where $s=\left[\begin{array}{ccc}
    p&v&z
\end{array}\right]^T$, $p,v$ are the position and velocity of the ego vehicle, respectively, and $z$ is the distance between the lead vehicle and the ego vehicle, $m$ is the mass of the ego vehicle, and $F_r(v)$ is the friction modelled as $F_r(v)=f_0+f_1v+f_2v^2$. The input of the ego vehicle $u$ which represents accelerator or brake is subject to
\begin{equation}
    -mc_dg\leq u \leq mc_ag
\end{equation}

Given a nominal controller of the ego vehicle $k(x)=F_r(v)$, we aim to achieve the desired velocity $v_d$ and keep a safe distance from the lead vehicle. The CLF and CBF functions are designed as follows:
\begin{equation}
    V(x)=(v-v_d)^2
\end{equation}
as a CLF and 
\begin{equation}
    h(x)=z-T_{h} v-\frac{1}{2} \frac{\left(v-v_{0}\right)^{2}}{c_{d} g}
\end{equation}
as a CBF, where $T_h$ is the look-ahead time.

We compare the results obtained from CLF-CBF QP, Optimal-decay CLF-CBF QP and Safety-first approach. Depending on the conditions whether the initial state is safe, i.e., $h(s_0)>0$ or not, and whether the initial speed is bigger than the desired speed, i.e., whether $v_0>v_d$ or not, four different scenarios are designed and the corresponding settings are summarized in \autoref{tab:ACC_restpara}. The hyper-parameters are fixed as shown in  \autoref{tab:ACC_fixedpara} in our simulations.

\subsubsection{Results Analysis}

As shown in \autoref{fig:ACC_case1}-\autoref{fig:ACC_case2}, 
in scenarios 1 and 2, feasible control solutions exist all the time and guarantee the safety, verifying our analysis in \autoref{subsec:safety} that all the methods are safety-guaranteed and the formulated optimization problems are feasible when $h(x)>0$. But there still exist  some differences:  
CLF decreases and CBF increases faster in the safety-first approach than that in other CLF-CBF QP methods. It verifies our analysis in \autoref{subsec:convergence}: compared with safety-first CLF-CBF QP, the introduced slack variables in both CLF-CBF QP and Optimal CLF-CBF QP degrade convergence rates.

\begin{figure}[hbtp]
    \centering
\includegraphics[scale=0.12]{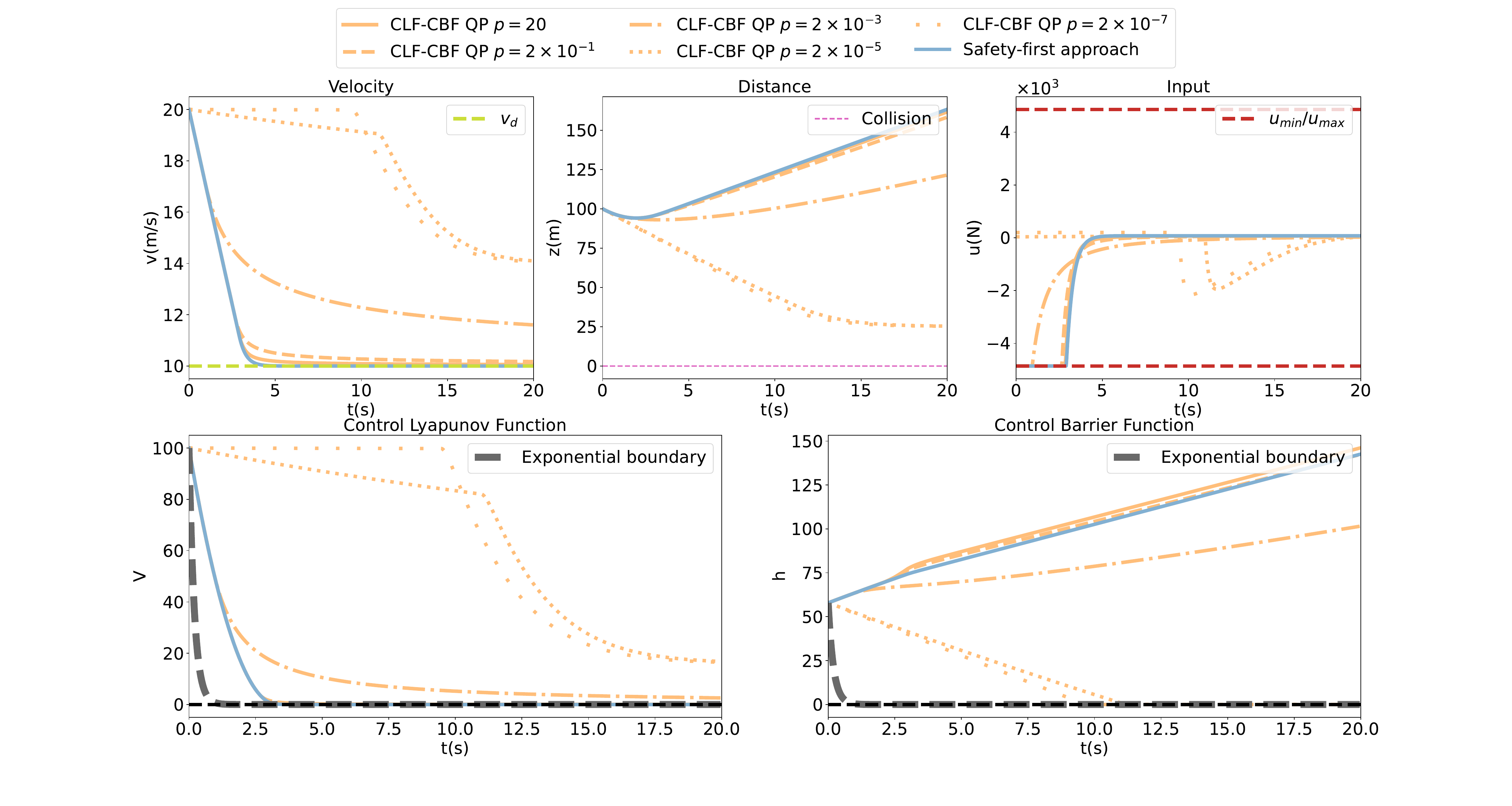}
    \caption{The result of the extended simulation. All CLF functions are above the black border determined by $\lambda$, which means CLF constraint is slacked in all approaches. As $p$ increases, CLF value decreases fast, and the ones of Safety-first CLF-CBF QP (blue) are the lower bounds.}
    \label{fig:ACC_extended}
\end{figure}

In scenarios 3 and 4, CLF-CBF QP is feasible only at the start, and then as the hard CBF constraint soon conflicts with input constraints, the CLF-CBF QP become infeasible since then. In contrast, Optimal-decay CLF-CBF QP and safety-first approach are feasible all the time, but Safety-first CLF-CBF QP yields faster CLF decrease and CBF increase. Particularly,  in scenario 4, the solution to Optimal-decay CLF-CBF QP leads to a constant decrease in CBF and hence collisions during the process, which verifies our analytical results in \autoref{subsec:safety}, that is Optimal-decay CLF-CBF QP can not guarantee the sub-safety of the system whereas safety-first approach can.

Furthermore, to verify our analysis results on convergence degradation in \autoref{subsec:convergence}, we implement both Safey-first CLF-CBF QP and CLF-CBF QP with varying weighting coefficient $p$ for Scenario 1. As shown in \autoref{fig:ACC_extended}, as the value of $p$ increases, the decay rate of CLF in CLF-CBF QP increases as well, approaching the decay rate of CLF in the safety-first approach as the limit.

\subsection{AGV Collision Avoidance }\label{subsec:2DDI}
\subsubsection{Problem Setup}\label{sub2sec:2DDIproblem}
The 2D double integrator is a simplified model of an AGV ignoring its geometric shape and physical dynamics. The system model is given below:
\begin{equation}
    s=\left[\begin{array}{c} x\\y\\v_x\\v_y
    \end{array}\right],\dot{s}=\left[\begin{array}{cccc}
    0&0&1&0\\
    0&0&0&1\\
    0&0&0&0\\
    0&0&0&0
    \end{array}
    \right]s+\left[\begin{array}{cc}
    0&0\\
    0&0\\
    1&0\\
    0&1
    \end{array}\right]u
\end{equation}
where $u=[u_1\,\, u_2]^T\in \mathbb{R}^2$, $p=[x\,\, y]^T$ is the position and $v=[v_x\,\,v_y]$ are the velocity along $x$ and $y$ respectively. 

\begin{table}[hbtp]
\caption{Parameters in AGV simulations}
\vspace{-0.4cm}
\label{tab:2DDI_para}
\begin{center}
\begin{threeparttable} 
\begin{tabular}{lcccccc}
    \toprule
    &Parameters&Value\\
    \midrule
    &$s_0$&$[0,4,0,0]^T$\\
    &$p_d$&$[10,0]$\\
    &$p_o$&$[5,3]$\\
    &$\rho$&$2$\\
    &$\lambda$&$1$\\
    &$\gamma$&$3$\\
    &$H$&$\left[\begin{array}{cc}
        5&0\\
        0&5 
    \end{array}\right]$\\

    &$p$&$1(0.01)^*$\\
    &$\omega_0$&$1$\\
    &$p_\omega$&$10(1)^*$\\
    &Control frequency (Hz)&50\\
    \bottomrule
\end{tabular}
      \begin{tablenotes} 
    \item $^*$ $(\cdot)$ means the value in the second simulation.
     \end{tablenotes} 
\end{threeparttable} 
\vspace{-0.4cm}
\end{center}
\end{table}

The input constraint is
\begin{equation}
   -7\leq u_{1,2}\leq 7
\end{equation}

We consider two different scenarios: scenario one involves only one circular obstacle with radius $\rho$ and position $p_o$, and scenario two involves multiple circular obstacles with different size and the details is presented in \autoref{tab:multi-obstacles}. The goal is to navigate to the point $p_d$ and avoid collision. Hence, the desired state $s_d=\left[\begin{array}{ccc}
p_d&0&0
\end{array}\right]^T$. 

\begin{figure}[ht]
\centering
\subfigure[CLF and CBF values]{
\includegraphics[scale=.15]{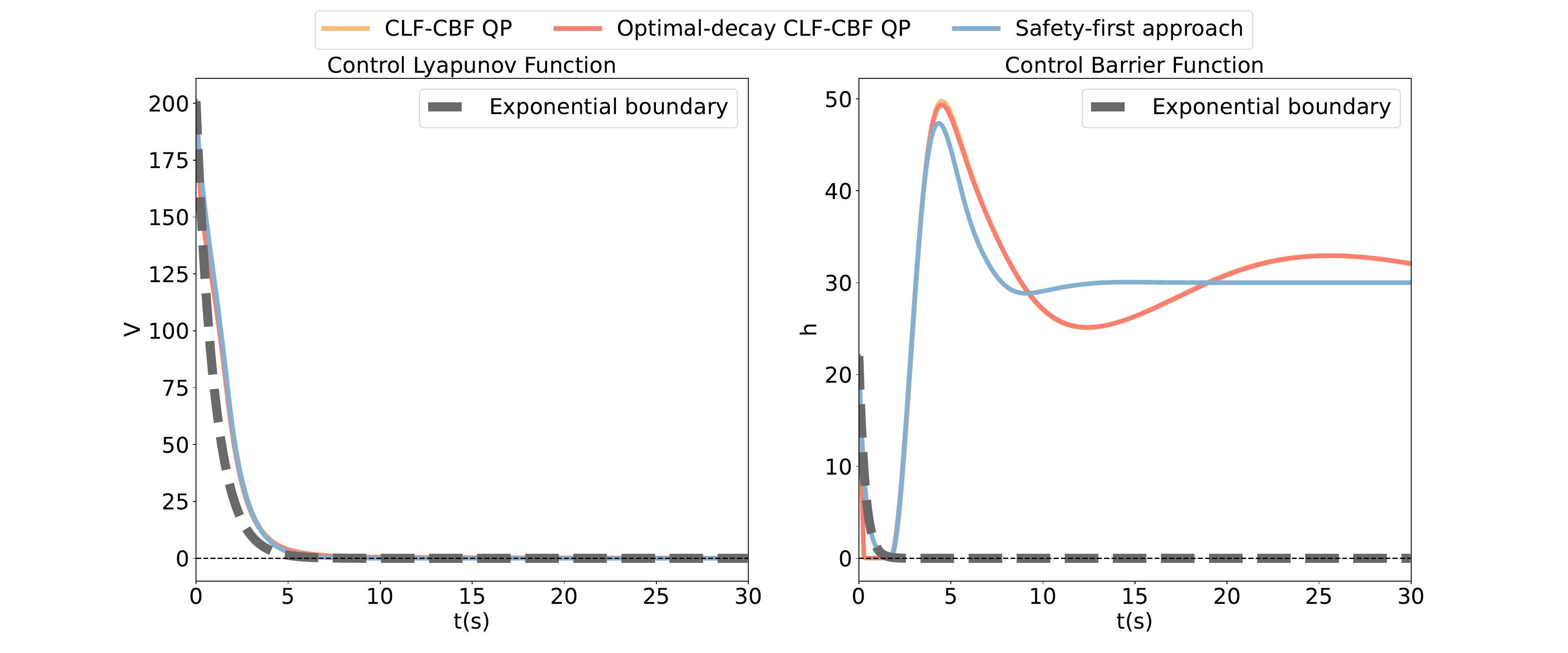}
    \label{fig:2DDIdata}}
\subfigure[Trajectories]{
\includegraphics[scale=.18]{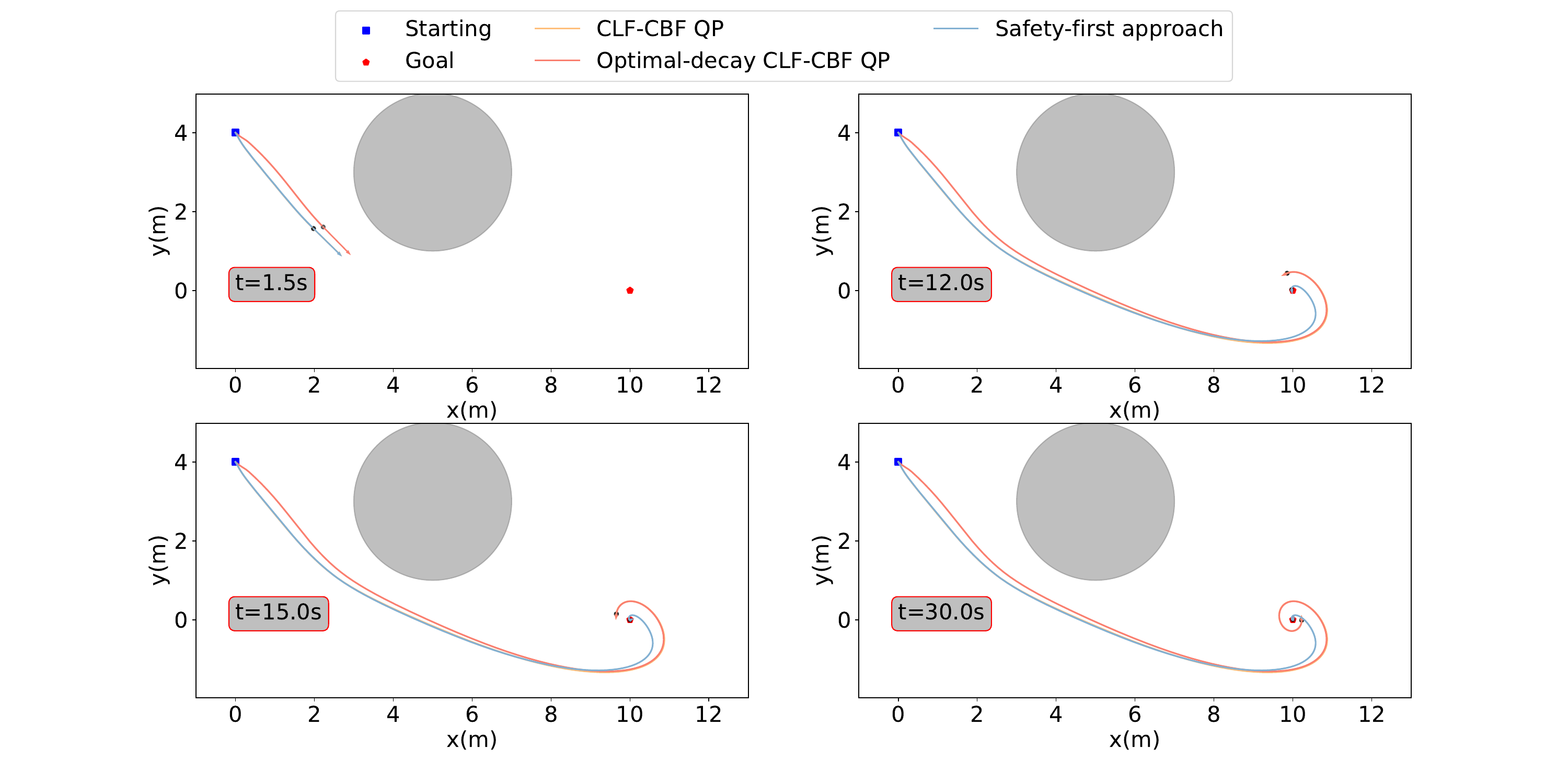}
    \label{fig:2DDItraj}}
\caption{Results of AGV collision avoidance simulation. All the methods avoid the obstacle, and Safety-first approach (blue) reaches the goal first.}
\end{figure}

We use the same CBF and CLF design as chosen in \cite{clf-cbf-helper}.  It uses linear quadratic regulator (LQR) to design a CLF. Assume $P$ is the solution of its algebraic Riccati equation, then the CLF is
\begin{equation}
    V=(s-s_d)^TP(s-s_d)
\end{equation}
and define CBF as 
\begin{equation}\label{cbfconstruct}
    h=(p-p_o)^T(p-p_o)-\rho^2+2(p-p_o)^Tv
\end{equation}
which means the sum of the square of distance from obstacle and its derivative.

\begin{figure}[hbt]
\centering
\subfigure[CLF and CBF values]{
\includegraphics[scale=.14]{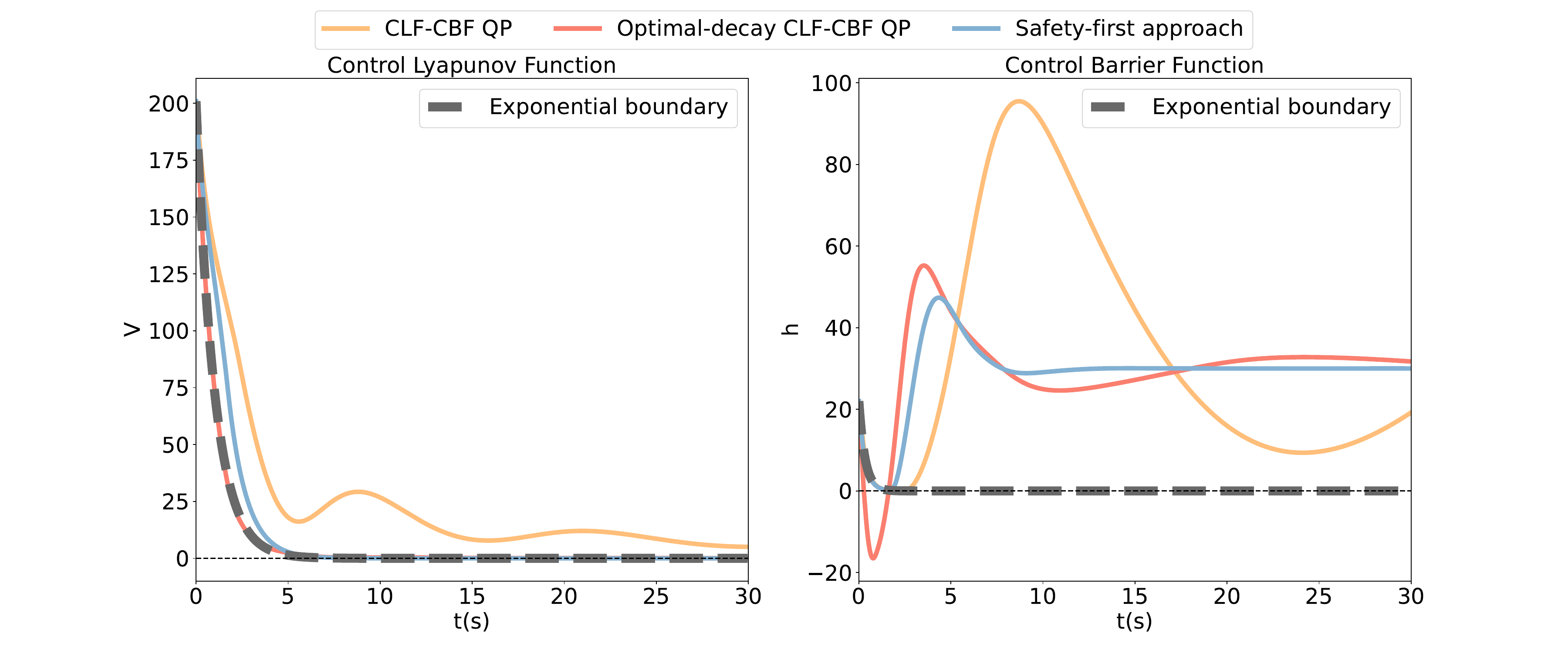}
    \label{fig:2DDIdata_slack}}
\subfigure[Trajectories]{
\includegraphics[scale=.14]{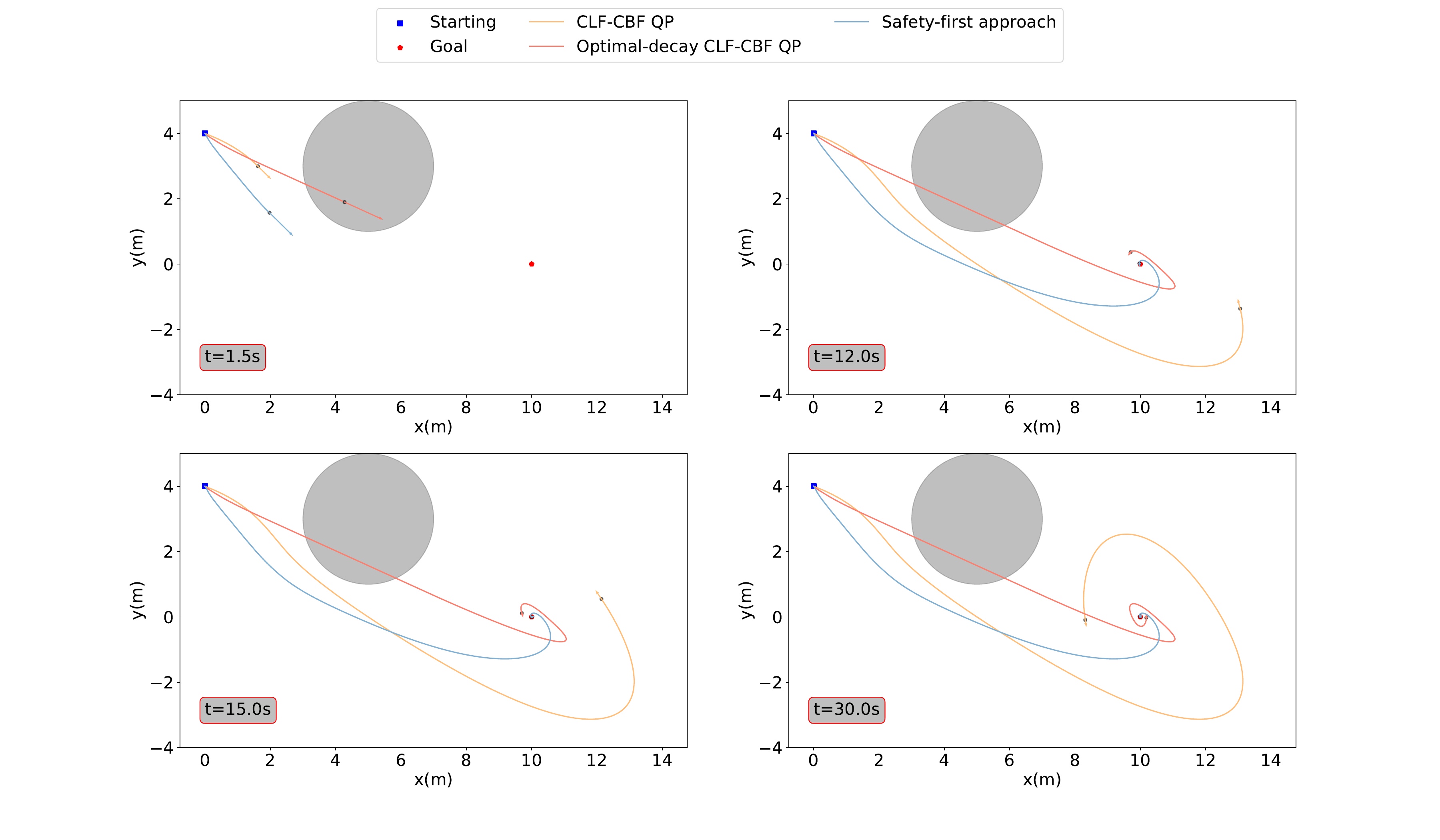}
    \label{fig:2DDItraj_slack}}
    \caption{Results of AGV collision avoidance simulation with smaller $p$ and $p_\omega$. Optimal-decay CLF-CBF QP (orange) drives the AGV towards the obstacle and collision, CLF-CBF QP (yellow) makes the AGV take a winding trajectory while Safety-first approach (blue) avoids collision and reaches the goal fastest.}
\end{figure}

\autoref{tab:2DDI_para} summarizes parameters set in simulations. All the parameters are fixed through the simulation except the values of weight $p$ and $p_\omega$. For comparative analysis, two sets of simulations with different values have been conducted: the first one with $p=1$ and $p_\omega=10$, and the second one with  $p=0.01$ and $p_\omega=1$.

\subsubsection{Results Analysis}\label{sub2sec:2DDIresults}
As shown in \autoref{fig:2DDIdata} and \autoref{fig:2DDItraj}, the controller derived from Optimal-decay CLF-CBF QP does not satisfy the exponential boundary of CBF at the start, but CLF-CBF QP and Safety-first approach do. Moreover, the trajectory of Safety-first approach reaches the goal point much quicker, since it satisfies the specified decay rate as much as possible. 

Then, we also have an extended simulation where $p=0.01$ and $p_\omega=1$, as shown in \autoref{fig:2DDIdata_slack} and \autoref{fig:2DDItraj_slack}. Compared with Safety-first approach, the convergence rate of CLF-CBF QP is slower, and Optimal-decay CLF-CBF QP can not guarantee safety any more. It also verifies our analysis in \autoref{subsec:safety} and \autoref{subsec:convergence}.

\subsubsection{Multi-obstacles cases}\label{sub2sec:multiobs}
In this part, we show that Safe-first CLF-CBF QP can also handle collision avoidance with multi-obstacles.

We construct a CBF for each obstacle following \eqref{cbfconstruct}. Priorities are determined by their CBF values: the lower the value is, the higher the priority is. All obstacles are set as \autoref{tab:multi-obstacles}, and other parameters are the same as \autoref{tab:2DDI_para}. As the trajectories shown in \autoref{fig:multiobs}, the AGV avoid obstacles successfully during navigation, verifying that Safety-first CLF-CBF QP works in multi-obstacles cases as well.

\begin{table}[hbtp]
\caption{Multi-obstacles}
\vspace{-0.4cm}
\label{tab:multi-obstacles}
\begin{center}
\begin{tabular}{lcccccc}
    \toprule
    &Obstacles&Center&Radius\\
    \midrule
    &1&[5,3]&1\\
    &2&[4,1]&1\\
    &3&[9,1]&1\\
    &4&[1,4]&0.5\\
    &5&[3,3]&0.5\\
    &6&[6,1]&0.3\\
    \bottomrule
    \vspace{-1cm}
\end{tabular}
\end{center}
\end{table}

\begin{figure}[hbt]
    \centering
\includegraphics[scale=.25]{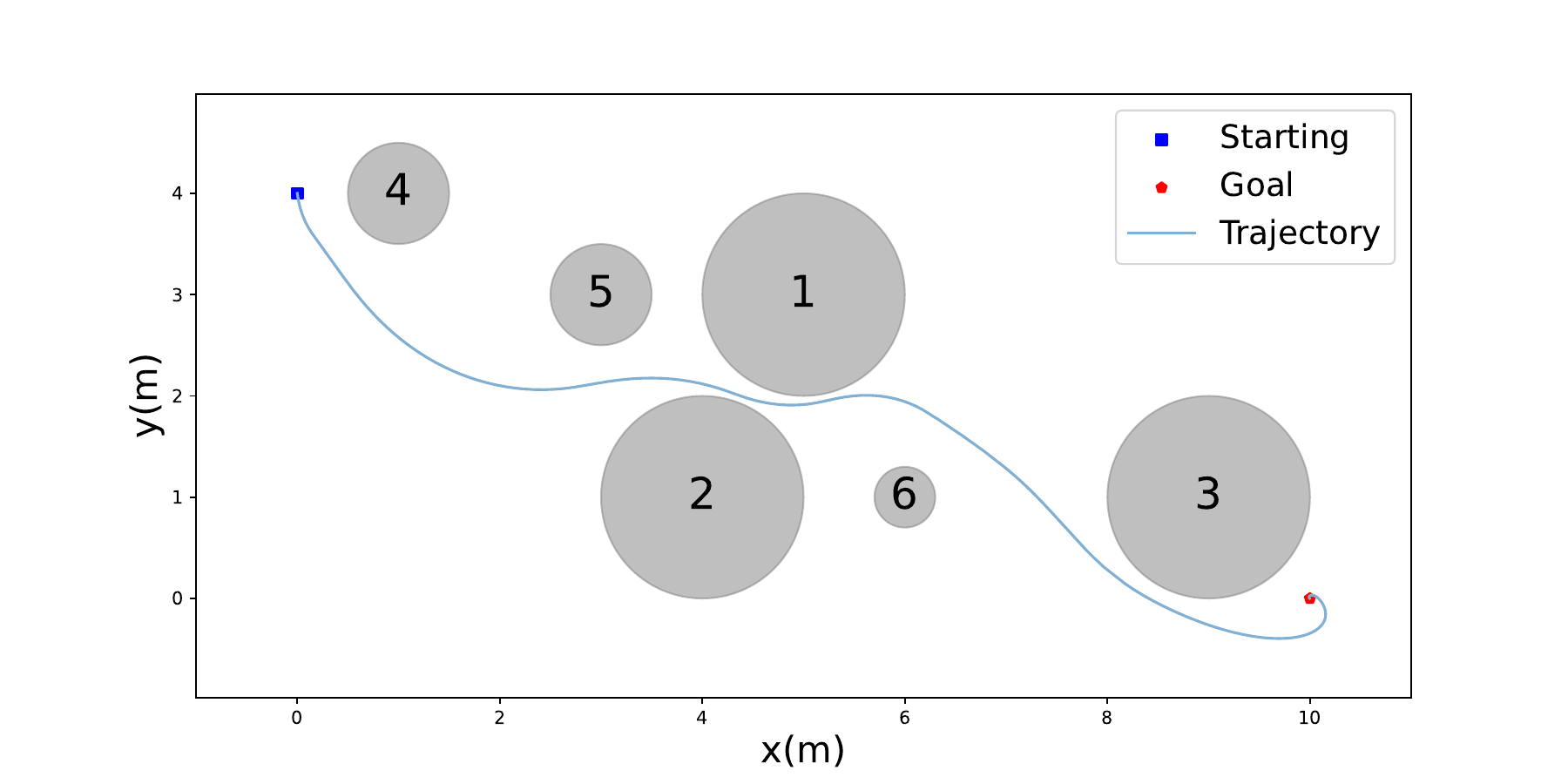}
    \caption{Multiple circular obstacles scenario. Safety-first CLF-CBF QP generates a smooth and collision-free path.}
    \label{fig:multiobs}
\end{figure}

\section{Conclusion}\label{sec:conclusion}
In this paper, we systematically analyse the shortcomings of existing CBF-based QP methods and develop the concept of sub-safety. Then, we propose a novel safety-first approach that addresses safety, stability and performance hierarchically and maintains feasibility all the time. Finally, various numerical simulations are provided to verify our analysis and the superiority of our proposed Safety-first CLF-CBF QP, and its flexibility in dealing with multi-certificate in multi-obstacles avoidance.





\bibliography{reference}

\begin{thebibliography}{10}
\providecommand{\url}[1]{#1}
\csname url@samestyle\endcsname
\providecommand{\newblock}{\relax}
\providecommand{\bibinfo}[2]{#2}
\providecommand{\BIBentrySTDinterwordspacing}{\spaceskip=0pt\relax}
\providecommand{\BIBentryALTinterwordstretchfactor}{4}
\providecommand{\BIBentryALTinterwordspacing}{\spaceskip=\fontdimen2\font plus
\BIBentryALTinterwordstretchfactor\fontdimen3\font minus \fontdimen4\font\relax}
\providecommand{\BIBforeignlanguage}[2]{{%
\expandafter\ifx\csname l@#1\endcsname\relax
\typeout{** WARNING: IEEEtran.bst: No hyphenation pattern has been}%
\typeout{** loaded for the language `#1'. Using the pattern for}%
\typeout{** the default language instead.}%
\else
\language=\csname l@#1\endcsname
\fi
#2}}
\providecommand{\BIBdecl}{\relax}
\BIBdecl

\bibitem{acc_CDC}
A.~D. Ames, J.~W. Grizzle, and P.~Tabuada, ``Control barrier function based quadratic programs with application to adaptive cruise control,'' in \emph{53rd IEEE Conference on Decision and Control}.\hskip 1em plus 0.5em minus 0.4em\relax IEEE, 2014, pp. 6271--6278.

\bibitem{QP_TAC}
A.~D. Ames, X.~Xu, J.~W. Grizzle, and P.~Tabuada, ``Control barrier function based quadratic programs for safety critical systems,'' \emph{IEEE Transactions on Automatic Control}, vol.~62, no.~8, pp. 3861--3876, 2016.

\bibitem{adaptiveCBF_ACC}
A.~J. Taylor and A.~D. Ames, ``Adaptive safety with control barrier functions,'' in \emph{2020 American Control Conference (ACC)}, 2020, pp. 1399--1405.

\bibitem{adaptiveCBF_XW}
W.~Xiao, C.~Belta, and C.~G. Cassandras, ``Adaptive control barrier functions,'' \emph{IEEE Transactions on Automatic Control}, vol.~67, no.~5, pp. 2267--2281, 2022.

\bibitem{robustCBF_YH}
S.~Kang, Y.~Chen, H.~Yang, and M.~Pavone, ``Verification and synthesis of robust control barrier functions: Multilevel polynomial optimization and semidefinite relaxation,'' in \emph{2023 62nd IEEE Conference on Decision and Control (CDC)}, 2023, pp. 8215--8222.

\bibitem{DOBCCBF_TAC}
J.~Sun, J.~Yang, and Z.~Zeng, ``Safety-critical control with control barrier function based on disturbance observer,'' \emph{IEEE Transactions on Automatic Control}, pp. 1--8, 2024.

\bibitem{nonsmoothcbf_tro}
R.~Funada, M.~Santos, R.~Maniwa, J.~Yamauchi, M.~Fujita, M.~Sampei, and M.~Egerstedt, ``Distributed coverage hole prevention for visual environmental monitoring with quadcopters via nonsmooth control barrier functions,'' \emph{IEEE Transactions on Robotics}, vol.~40, pp. 1546--1565, 2024.

\bibitem{nonsmoothcbf_cdc}
L.~Wang, A.~D. Ames, and M.~Egerstedt, ``Multi-objective compositions for collision-free connectivity maintenance in teams of mobile robots,'' in \emph{2016 IEEE 55th Conference on Decision and Control (CDC)}, 2016, pp. 2659--2664.

\bibitem{nonsmoothcbf_ral}
P.~Glotfelter, I.~Buckley, and M.~Egerstedt, ``Hybrid nonsmooth barrier functions with applications to provably safe and composable collision avoidance for robotic systems,'' \emph{IEEE Robotics and Automation Letters}, vol.~4, no.~2, pp. 1303--1310, 2019.

\bibitem{nonsmoothcbf_ccta}
P.~Glotfelter, J.~Cortés, and M.~Egerstedt, ``Boolean composability of constraints and control synthesis for multi-robot systems via nonsmooth control barrier functions,'' in \emph{2018 IEEE Conference on Control Technology and Applications (CCTA)}, 2018, pp. 897--902.

\bibitem{RLCBF_RSS}
J.~Choi, F.~Casta{\~n}eda, C.~J. Tomlin, and K.~Sreenath, ``Reinforcement learning for safety-critical control under model uncertainty, using control lyapunov functions and control barrier functions,'' in \emph{Robotics: Science and Systems (RSS)}, 2020.

\bibitem{barriernet_TRO}
W.~Xiao, T.-H. Wang, R.~Hasani, M.~Chahine, A.~Amini, X.~Li, and D.~Rus, ``Barriernet: Differentiable control barrier functions for learning of safe robot control,'' \emph{IEEE Transactions on Robotics}, vol.~39, no.~3, pp. 2289--2307, 2023.

\bibitem{expertdata_CDC}
A.~Robey, H.~Hu, L.~Lindemann, H.~Zhang, D.~V. Dimarogonas, S.~Tu, and N.~Matni, ``Learning control barrier functions from expert demonstrations,'' in \emph{2020 59th IEEE Conference on Decision and Control (CDC)}, 2020, pp. 3717--3724.

\bibitem{survey_TRO}
C.~Dawson, S.~Gao, and C.~Fan, ``Safe control with learned certificates: A survey of neural lyapunov, barrier, and contraction methods for robotics and control,'' \emph{IEEE Transactions on Robotics}, vol.~39, no.~3, pp. 1749--1767, 2023.

\bibitem{S2RL_TIV}
B.~Gangopadhyay, P.~Dasgupta, and S.~Dey, ``Safe and stable rl (s2rl) driving policies using control barrier and control lyapunov functions,'' \emph{IEEE Transactions on Intelligent Vehicles}, vol.~8, no.~2, pp. 1889--1899, 2023.

\bibitem{cbf4rob_icra1}
D.~Du, S.~Han, N.~Qi, H.~B. Ammar, J.~Wang, and W.~Pan, ``Reinforcement learning for safe robot control using control lyapunov barrier functions,'' in \emph{2023 IEEE International Conference on Robotics and Automation (ICRA)}, 2023, pp. 9442--9448.

\bibitem{cbf4rob_ral1}
J.~Li, Q.~Liu, W.~Jin, J.~Qin, and S.~Hirche, ``Robust safe learning and control in an unknown environment: An uncertainty-separated control barrier function approach,'' \emph{IEEE Robotics and Automation Letters}, vol.~8, no.~10, pp. 6539--6546, 2023.

\bibitem{cbf4rob_icra2}
M.~Tong, C.~Dawson, and C.~Fan, ``Enforcing safety for vision-based controllers via control barrier functions and neural radiance fields,'' in \emph{2023 IEEE International Conference on Robotics and Automation (ICRA)}, 2023, pp. 10\,511--10\,517.

\bibitem{cbf4rob_icra3}
H.~Abdi, G.~Raja, and R.~Ghabcheloo, ``Safe control using vision-based control barrier function (v-cbf),'' in \emph{2023 IEEE International Conference on Robotics and Automation (ICRA)}, 2023, pp. 782--788.

\bibitem{optimal-decay_ACC}
J.~Zeng, B.~Zhang, Z.~Li, and K.~Sreenath, ``Safety-critical control using optimal-decay control barrier function with guaranteed point-wise feasibility,'' in \emph{2021 American Control Conference (ACC)}.\hskip 1em plus 0.5em minus 0.4em\relax IEEE, 2021, pp. 3856--3863.

\bibitem{sufficient_Automatica}
W.~Xiao, C.~A. Belta, and C.~G. Cassandras, ``Sufficient conditions for feasibility of optimal control problems using control barrier functions,'' \emph{Automatica}, vol. 135, p. 109960, 2022.

\bibitem{checking_CDC2022}
X.~Tan and D.~V. Dimarogonas, ``Compatibility checking of multiple control barrier functions for input constrained systems,'' in \emph{2022 IEEE 61st Conference on Decision and Control (CDC)}, 2022, pp. 939--944.

\bibitem{2024_Automatica}
Y.~Dong, X.~Wang, and Y.~Hong, ``Safety critical control design for nonlinear system with tracking and safety objectives,'' \emph{Automatica}, vol. 159, p. 111365, 2024.

\bibitem{withconstraints_LCSS2022}
W.~Shaw~Cortez, X.~Tan, and D.~V. Dimarogonas, ``A robust, multiple control barrier function framework for input constrained systems,'' \emph{IEEE Control Systems Letters}, vol.~6, pp. 1742--1747, 2022.

\bibitem{ESCLF_TAC}
A.~D. Ames, K.~Galloway, K.~Sreenath, and J.~W. Grizzle, ``Rapidly exponentially stabilizing control lyapunov functions and hybrid zero dynamics,'' \emph{IEEE Transactions on Automatic Control}, vol.~59, no.~4, pp. 876--891, 2014.

\bibitem{theoryandapp_ECC}
A.~D. Ames, S.~Coogan, M.~Egerstedt, G.~Notomista, K.~Sreenath, and P.~Tabuada, ``Control barrier functions: Theory and applications,'' in \emph{2019 18th European Control Conference (ECC)}.\hskip 1em plus 0.5em minus 0.4em\relax IEEE, 2019, pp. 3420--3431.

\bibitem{PlanarQuad_ACC}
G.~Wu and K.~Sreenath, ``Safety-critical control of a planar quadrotor,'' in \emph{2016 American Control Conference (ACC)}, 2016, pp. 2252--2258.

\bibitem{clf-cbf-helper}
J.~Choi, P.~Kotaru, and B.~Zhang, ``Cbf-clf-helper,'' https://github.com/HybridRobotics/CBF-CLF-Helper.

\end{thebibliography}
\bibliographystyle{IEEEtran}

\end{document}